\documentclass[a4paper,11pt]{amsart}
\usepackage{amsfonts, amssymb, amsmath}
\usepackage{enumerate,mathrsfs}
\usepackage{latexsym,amssymb,eucal}
\usepackage[pdfpagelabels,colorlinks,linkcolor=blue,citecolor=black,urlcolor=blue]{hyperref}
\newtheorem{thm}{Theorem}[section]

\newtheorem{lem}[thm]{Lemma}
\newtheorem{prop}[thm]{Proposition}

\theoremstyle{remark}
\newtheorem{rem}{Remark}[section]

\newcommand{\al}{\alpha}
\newcommand{\vc}{\infty}

\newcommand{\f}{\frac}
\newcommand{\les}{\lesssim}

\newcommand{\bra}[1]{\langle #1 \rangle}
\newcommand{\bbar}{|\!|\!|}

\title[SHARP $L^p$ ESTIMATES FOR SCHR\"ODINGER GROUPS]{SHARP $L^p$ ESTIMATES FOR SCHR\"ODINGER GROUPS ON SPACES OF HOMOGENEOUS TYPE}         

\author{The Anh Bui}
\address{The Anh Bui: Department of Mathematics, Macquarie University, NSW 2109, Australia}

\email{the.bui@mq.edu.au, bt\_anh80@yahoo.com}

\author{Piero D'ancona}
\address{Piero D'ancona: Sapienza – Universit\`a di Roma, Dipartimento di Matematica, Piazzale A. Moro   2, I-00185 Roma, Italy}

\email{dancona@mat.uniroma1.it}

\author{Fabio Nicola}
\address{Fabio Nicola: Dipartimento di Scienze Matematiche, Politecnico di Torino, Corso Duca degli Abruzzi 24, 10129 Torino, Italy}

\email{fabio.nicola@polito.it}

\thanks{The Anh Bui was supported by the research grant ARC DP 140100649}

\keywords{%
Schr\"{o}dinger group%
; metric measure spaces%
; doubling measure%
; spectral multipliers%
; heat kernels}
\subjclass[2010]{%
35P25
, 35Q41
, 42B15
, 35K08
}

\begin{document}

\date{}
\begin{abstract}
  We prove an $L^{p}$ estimate
  \begin{equation*}    
    \|e^{-itL}
    \varphi(L)f\|_{p}\les (1+|t|)^s\|f\|_p,
    \qquad t\in \mathbb{R},
    \qquad
    s=n\left|\f{1}{2}-\f{1}{p}\right|
  \end{equation*}
  for the Schr\"{o}dinger group generated by a 
  semibounded, self--adjoint 
  operator $L$ on a metric measure space $\mathcal{X}$ of
  homogeneous type (where $n$ is the doubling dimension
  of $\mathcal{X}$). The assumptions on $L$ are
  a mild $L^{p_{0}}\to L^{p_{0}'}$ smoothing estimate
  and a mild $L^{2}\to L^{2}$ off--diagonal estimate
  for the corresponding heat kernel $e^{-tL}$. The estimate is 
  uniform for $
  \varphi$ varying
  in bounded sets of $\mathscr{S}(\mathbb{R})$,
  or more generally of a suitable weighted Sobolev space.

  We also prove, under slightly stronger assumptions on $L$,
  that the estimate extends to
  \begin{equation*}
    \|e^{-itL}
    \varphi(\theta L)f\|_{p}\les (1+\theta^{-1}|t|)^s\|f\|_p,
    \qquad \theta>0, \quad t\in \mathbb{R},
  \end{equation*}
  with uniformity also for $\theta$ varying in bounded subsets
  of $(0,+\infty)$. For nonnegative operators uniformity
  holds for all $\theta>0$.
\end{abstract}

\maketitle


\section{Introduction}

Bounds in $L^{p}$ for the Schr\"{o}dinger group $e^{it \Delta}$
have applications in harmonic
analysis and to nonlinear dispersive equations. The group
itself is not bounded in $L^{p}$ for $p\neq2$,
but $(1-\Delta)^{-s}e^{it \Delta}$ is $L^{p}$ bounded for
$s$ sufficiently large. A sharp estimate can be written
if one introduces a frequency cutoff
$\varphi\in C^{\infty}_{c}(\mathbb{R}^{n})$:
for all $1\leq p\leq\infty$, 
$k\in\mathbb{Z}$, $t\in \mathbb{R}$, we have
\begin{equation}\label{zero}
  \|e^{it\Delta}\varphi(2^{-k}(-\Delta)^{1/2}) f\|_{L^p}\lesssim 
  (1+2^{2k}|t|)^{s}\|f\|_{L^p},
  \qquad
  s=n\biggl|\frac{1}{2}-\frac{1}{p}\biggr|,
\end{equation}
see 
\cite{BrennerThomeeWahlbin75-a}, 
\cite{Sjostrand70-a},
\cite{Lanconelli68-b}.

This result can be regarded as an elementary example
of $L^{p}$ estimates with loss of derivatives for FIOs,
in the spirit of \cite{SeegerSoggeStein91-a}. However, 
our goal here is to extend \eqref{zero} in
a different direction, namely, to Schr\"{o}dinger groups 
$e^{itL}$ generated by
a semibounded, self--adjoint operator $L$ on 
a metric measure
space $\mathcal{X}$ endowed with a doubling measure. 
This framework covers a large variety of
situations which go far beyond the classical FIO setting.

Many properties of $L$ and functions of $L$ can be deduced
from suitable estimates on the corresponding heat kernel $e^{-tL}$.
A common assumption in the euclidean case 
(see \cite{Hebisch90-b}, \cite{DuongOuhabazSikora02-a})
is the Gaussian upper estimate
\begin{equation*}
  |e^{-tL}(x,y)|\lesssim t^{-n/m} 
  \exp\Big(-b\big(t^{-1/m}|x-y|\big)^{\frac{m}{m-1}}\Big),
  \quad \ t>0,\ x,y\in \mathbb{R}^{n},
\end{equation*}
for some $b>0$, $m>1$. 
This includes Schr\"{o}dinger operators perturbed with
an electromagnetic potential (in this case $m=2$: see
\cite{CacciafestaDAncona12-a}, \cite{CassanoDAncona15-a} for
some applications), and fractional Laplacians $(-\Delta)^{m/2}$ with $m$ even.
Note that these operators are already outside the reach of the
classical theory of singular operators.

In order to include more general operators, one
can weaken the assumptions on the heat kernel. In
\cite{DAnconaNicola15-a} we proposed, in the euclidean case,
to replace the Gaussian upper estimate with a weak 
$L^{p_{0}}\to L^{p_{0}'}$ smoothing estimate on dyadic cubes,
and an even weaker off--diagonal $L^{2}\to L^{2}$ algebraic decay
(see \eqref{eq1-assumption}, \eqref{eq2-assumption} below).
These conditions are much more inclusive, as discussed in
Remark \ref{rem2} below, but they still allow to recover the 
estimate \eqref{zero} at least in the restricted range 
$p\in[p_{0},p_{0}']$.

Here we study the more general situation of metric
measure spaces of homogeneous type.
More precisely, in the following we shall assume that
$(\mathcal{X},d,\mu)$ is a metric space with distance $d$,
equipped with a
nonnegative Borel measure $\mu$ which satisfies the 
\emph{doubling property}: there exists a constant $c_1>0$ 
such that
\begin{equation}\label{doublingpro}
  \mu(B(x,2r))\leq c_1\mu(B(x,r))
\end{equation}
for all $x\in \mathcal{X}$ and $r>0$, where $B(x, r)$ is
the open ball of radius $r$ and center $x$.
We recall that the doubling property (\ref{doublingpro})
implies the existence of $C>0$ and $n>0$ such that
\begin{equation*}
  \mu(B(x,\lambda r))\leq C\lambda^n \mu(B(x,r)),
  \qquad
  \forall\lambda>0.
\end{equation*}
We shall also assume that $\mathcal{X}$ satisfies
a \emph{reverse doubling condition}:
there exist $\kappa\in[0,n]$ and $C>0$ such that for all 
$x\in \mathcal{X}$, $0<r<{\rm diam}(\mathcal{X})/2$ 
and $1\leq \lambda<{\rm diam}(\mathcal{X})/(2r)$, one has
	\begin{equation}
	\label{RD-eq}
	C\lambda^{\kappa}\mu(B(x,r))\leq \mu(B(x,\lambda r))
	\end{equation}
	where ${\rm diam}(\mathcal{X})=\sup_{x,y\in \mathcal{X}}d(x,y)$.
Note that the reverse doubling condition is always 
satisfied with $\kappa=0$,
thus \eqref{RD-eq} is restrictive only when $\kappa\in(0,n]$.

It was proved in \cite{C} that it is always possible,
for each $\nu\in \mathbb{Z}$, to define an almost covering 
$\mathcal{D}_{\nu}$ of open sets, with diameter $\simeq 2^{-\nu}$,
which are called \emph{dyadic cubes} and enjoy
properties very similar to the standard dyadic cubes 
in $\mathbb{R}^{n}$;
see Lemma \ref{Christ'slemma} below for precise definitions
and more details.

In this setting, we consider an operator $L$ on
$L^{2}(\mathcal{X})$ satisfying the following assumption,
where $1_{Q}$ denotes the characteristic function of the
cube $Q$:

\medskip

\noindent \textbf{Assumption (L$_{0}$):} \ \ 
$L$ is a self--adjoint operator on $L^{2}(\mathcal{X})$ 
with $L+M_{0}\ge0$ for some constant $M_{0}\ge0$, 
satisfying the following estimate.
There exist $p_0\in [1,2)$, $m_1, m_2>0$ and $C\ge0$
such that for all $t>0$ and $\nu\in\mathbb{Z}$
with either 
$2^{-\nu}\leq t^{1/m_1}<2^{-\nu+1}$, $0<t<1$ or 
$2^{-\nu}\leq t^{1/m_2}<2^{-\nu+1}$, $t\geq 1$ 
we have
\begin{equation}\label{eq1-assumption}
  \sum_{Q\in \mathcal{D}_\nu}\|1_Q e^{-tL}1_{Q'}\|_{p_0\to 2}
  +
  \sum_{Q\in \mathcal{D}_\nu}\|1_Q e^{-tL}1_{Q'}\|_{2\to p_0'}
  \le
  Ce^{M_{0}t}
  (t_{>1}^{\f{n}{m_{1}}}+ 
   t_{\le1}^{\f{\kappa}{m_{2}}})^{\f{1}{2}-\f{1}{p_{0}}}
\end{equation}
for all $Q'\in \mathcal{D}_\nu$, where
$t_{>1}=t\cdot 1_{(1,+\infty)}(t)$ and
$t_{\le1}=t \cdot 1_{(-\infty,1]}(t)$,
and
\begin{equation}\label{eq2-assumption}
\sup_{Q'\in \mathcal{D}_\nu}\sum_{Q\in \mathcal{D}_\nu}
  (1+2^\nu {\rm dist}(Q,Q'))^{N}\|1_Q e^{-tL}1_{Q'}\|_{2\to 2}
  \leq Ce^{M_{0}t}, \qquad
  N=\lfloor n/2\rfloor+1.
\end{equation}

\begin{rem} \label{rem2}
  The previous assumptions include of course
  the typical Gaussian upper estimates for Schr\"{o}dinger operators
  on $\mathbb{R}^{n}$. Indeed,
  in the particular case $m_1=m_2=m$, condition 
  \eqref{eq2-assumption} is a direct consequence of the 
  following estimate
  \begin{equation}\label{Lpqestimates}
    \textstyle
    \|1_{B(x,t^{1/m})}e^{-tL}1_{B(y,t^{1/m})}\|_{p_0\to p_0'}
    \leq Ce^{M_{0}t}\mu(B(x,t^{1/m}))^{-\f{1}{p_0}+\f{1}{p'_0}}
    \exp\Big(-\f{d(x,y)^{m/(m-1)}}{ct^{1/(m-1)}}\Big),
  \end{equation}
  for all $t>0$, and all $x,y\in \mathcal{X}$. Note that the
  presence of an exponentially growing factor $e^{M_{0}t}$
  allows to include some interesting cases like 
  non--positive Schr\"{o}dinger operators $-\Delta+V(x)$,
  see \cite{Simon82-a}.

  However, the converse implication may be false. 
  An example is given by the fractional
  Laplacian $L=(-\Delta)^\alpha, \alpha>0$. It was proved in 
  \cite{DAnconaNicola15-a} 
  that for $\alpha>\lfloor n/2\rfloor+1$, $L$ 
  satisfies \eqref{eq1-assumption} and \eqref{eq2-assumption} 
  for $p_0=1$ and $m=2\alpha$, but not \eqref{Lpqestimates}.
  
  Moreover, the estimate \eqref{Lpqestimates} does not imply the 
  condition \eqref{eq1-assumption}. However, if we assume in 
  addition that $(X,d,\mu)$ satisfies the 
  {\it non-collapsing} condition
  \begin{equation}\label{non-collapsing-app1}
    \mu(B(x,1))\gtrsim 1, \ \ \forall x\in \mathcal{X},
  \end{equation}
  then \eqref{eq1-assumption} is a consequence of 
  \eqref{Lpqestimates}. 
\end{rem}

Then we can prove:

\begin{thm}[]\label{the:first}
  Assume $L$ satisfies \textbf{(L$_{0}$)}.
  Let $p\in[p_{0},p_{0}']$ and
  $s=n\left|\f{1}{2}-\f{1}{p}\right|$.
  Then the estimate
  \begin{equation*}    
    \|e^{-itL}\varphi(L)f\|_{p}\les (1+|t|)^s\|f\|_p,
    \qquad t\in \mathbb{R}
  \end{equation*}
  holds uniformly for $\varphi$ in bounded subsets of
  $\mathscr{S}(\mathbb{R})$.
\end{thm}

\begin{rem}[]\label{rem:weightedsob}
  The previous result is still valid for functions
  $\varphi$ of Sobolev regularity. 
  More precisely, the estimate is true and
  uniform in $\varphi$ provided the following norm
  \begin{equation}\label{eq:weightedsob0}
    \sum_{j\le n+1}
    \|\bra{\lambda}^{2+n+j+n/m_{1}}\varphi^{(j)}(\lambda)\|_{L^{2}}.
  \end{equation}
  remains bounded. This condition is not sharp; see
  Remark \ref{rem:constant} for further details.
\end{rem}

We now examine a few directions in which one can relax the
assumptions of Theorem \ref{the:first}.
In order to do this we introduce some definitions.
The \emph{amalgam space} $X^{1,p}_{\nu}$,
with $1\le p\le \infty$ and $\nu\in \mathbb{Z}$, is 
the space of measurable functions on $\mathcal{X}$
such that the following norm is finite:
\begin{equation}\label{eq:amalgam}
  \|f\|_{X_{\nu}^{1,p}}:=\sum_{Q\in \mathcal{D}_{\nu}}
  \|f\|_{L^{p}(Q)}.
\end{equation}
Moreover, we say that
$w:\mathcal{X}\times \mathcal{X}\to \mathbb{R}$ 
is a \emph{weight function}
if it is equivalent to the distance function, in the sense that
\begin{equation}\label{eq:weight}
  \textstyle
  K_{0}^{-1}d(x,y)\le |w(x,y)|\le K_{0}d(x,y)
\end{equation}
for some constant $K_{0}>0$.

Let $w$ be a weight function and let $\mathscr D(w)$ be any topological vector space associated to $w$ satisfying the following conditions:
\begin{enumerate}[(i)]
	\item $\mathscr D(w)$ is dense in $L^2(X)$ (w.r.t. the $L^2(X)$ norm);
	\item $w(x,\cdot)^N f\in \mathscr D(w)$ for all $f\in \mathscr D(w)$, $x\in \mathcal{X}$ and $N\in \mathbb{N}$. 
\end{enumerate}
Denote by $\mathscr D'(w)$ the dual space of $\mathscr D(w)$. 

\begin{rem}
	In applications, in Subsections 3.1--3.10 we will choose $w(x,y)=d(x,y)$ and $\mathscr D(w)=L^2_c(X)$ which is a space of all functions in $L^2$ with compact support. In Subsection \ref{Subsec11}, as $X= \mathbb{R}^n$ we will choose $\mathscr D(w)=C_0^\vc(\mathbb{R}^n)$. We will do not recall this in Section 3.
\end{rem}

Denoting by $w_{x}$ the multiplication operator by
the function $w(x,\cdot)$, 
the commutators ${\rm Ad}_{x}^{k+1}(T): \mathscr D_X \to \mathscr D'_X$ of order $k$ of an $L^2(X)$-bounded linear operator $T$ with the weight $w$ are defined as follows:
\begin{equation*}
  {\rm Ad}_{x}^{0}(T)=I,
  \qquad
  {\rm Ad}_{x}^{1}(T)=[w_{x},T],
  \qquad
  {\rm Ad}_{x}^{k+1}(T)=[w_{x},{\rm Ad}_{x}^{k}(T)].
\end{equation*}
In view of the applications, we shall also consider a more general
kind of vector valued weight functions
$w=(w_{1},\dots,w_{\ell}):
  \mathcal{X}\times \mathcal{X}\to \mathbb{R}^{\ell}$,
defined again by condition \eqref{eq:weight} 
(where now $|w|=(w_{1}^{2}+\dots+w_{\ell}^{2})^{1/2}$).
In the vector valued case
${\rm Ad}_{x}^{k}(T)$ will denote the $\ell$--tuple of commutators
with $w_{1},\dots,w_{\ell}$, that is to say we define for
$j=1,\dots,\ell$
\begin{equation*}
  {\rm Ad}_{j,x}^{0}(T)=I,
  \qquad
  {\rm Ad}_{j,x}^{1}(T)=[w_{j,x},T],
  \qquad
  {\rm Ad}_{j,x}^{k+1}(T)=[w_{j,x},{\rm Ad}_{x}^{k}(T)].
\end{equation*}
(where $w_{j,x}$ is multiplication by $w_{j}(x,\cdot)$)
and 
${\rm Ad}_{x}^{k}(T):
  =({\rm Ad}_{1,x}^{k}(T),\dots,{\rm Ad}_{\ell,x}^{k}(T))$.
Note that the simplest choice of a weight satisfying
\eqref{eq:weight} is given by
the distance function itself, with $\ell=1$.

We can now state our second set of assumptions on $L$:
\medskip

\noindent \textbf{Assumption (L):} \ \ 
$L$ is a self--adjoint operator on
$L^{2}(\mathcal{X})$ with $L+M_{0}\ge0$ for some 
constant $M_{0}\ge0$,
satisfying the following estimates. 
There exist $p_0\in [1,2)$, $m_1, m_2>0$ and 
$C\ge0$ such that 
for all $t>0$ and $\nu\in\mathbb{Z}$ 
with either 
$2^{-\nu}\leq t^{1/m_1}<2^{-\nu+1}$, $0<t<1$ or 
$2^{-\nu}\leq t^{1/m_2}<2^{-\nu+1}$, $t\geq 1$ 
we have
\begin{equation}\label{eq1-assumption-0}
  \|e^{-tL}\|_{X_{\nu}^{1,p_{0}}\to X_{\nu}^{1,2}}
  +
  \|e^{-tL}\|_{X_{\nu}^{1,2}\to X_{\nu}^{1,p_{0}'}}
  \le
  Ce^{M_{0}t}
  (t_{>1}^{\f{n}{m_{1}}}+
    t_{\le1}^{\f{\kappa}{m_{2}}})^{\f{1}{2}-\f{1}{p_{0}}}
\end{equation}
 where
$t_{>1}=t\cdot 1_{(1,+\infty)}(t)$ and
$t_{\le1}=t \cdot 1_{(-\infty,1]}(t)$.
Moreover, there exists a weight function $w(x,y)$ 
and a constant $M_{1}>M_{0}$ such that
the resolvent $R(z)=(L+z)^{-1}$ satisfies,
for all $x\in \mathcal{X}$,
\begin{equation}\label{eq2-assumption-0}
  \|{\rm Ad}^{k}_{x}(R(M_{1}))\|_{2\to2}\le C
  \qquad
  0\le k\le\lfloor n/2\rfloor+1.
\end{equation}


\begin{rem}[]\label{rem:commuthyp}
  The reason why condition \eqref{eq2-assumption-0} is
  interesting, besides being much weaker than
  \eqref{eq2-assumption}, is that it is very easy
  to check directly for differential operators, and even some
  pseudodifferential ones, in the euclidean setting. See Subsection \ref{Subsec11}.
\end{rem}

Then we can prove:

\begin{thm}[]\label{the:second}
  Assume $L$ satisfies \textbf{(L)}.
  Let $p\in [p_0,p_0']$  and let $s=n\Big|\f{1}{2}-\f{1}{p}\Big|$.
  Then the estimate
  \begin{equation*}    
    \|e^{-itL}\varphi(L)f\|_{p}\les 
    (1+|t|)^s\|f\|_p,
    \qquad t\in \mathbb{R}
  \end{equation*}
  holds uniformly for $\varphi$ in bounded subsets of
  $\mathscr{S}(\mathbb{R})$.
\end{thm}

\begin{rem}[] \label{rem:LL}
  Comparing the two sets of assumptions we see that
  \begin{equation*}
    \ \text{Assumption \textbf{(L$_{0}$)}}\ 
    \implies
    \ \text{Assumption \textbf{(L)}.}\ 
  \end{equation*}
  Indeed, the implication
  \begin{equation*}
    \text{condition \eqref{eq1-assumption} $\implies$ 
    condition \eqref{eq1-assumption-0}}
  \end{equation*}
  is obviously true. On the other hand, one has
  \begin{equation*}
    \text{condition \eqref{eq2-assumption} $\implies$
    condition \eqref{eq2-assumption-0},
    with $w(x,y)=d(x,y)$,}
  \end{equation*}
  but this is more delicate and
  will be proved in Propositions \ref{pro:weakerL} 
  and \ref{pro:commR} below. 
\end{rem}

One notices that estimate \eqref{zero} for the standard Laplacian
is uniform also for
rescaling in frequency $\sim 2^{k}$, $k\in \mathbb{Z}$.
This is a direct consequence of the scaling properties
of $\mathbb{R}^{n}$ and its Lebesgue
measure, which are not available on a general metric measure
space $\mathcal{X}$. Uniformity in frequency is an important
property, especially useful when doing dyadic analysis on
Sobolev or Besov spaces generated by the operator $L$.
We can recover uniformity under slightly stronger assumptions on 
the operator $L$:

\medskip

\noindent \textbf{Assumption (L$_{1}$):} \ \ 
$L$ is a self--adjoint operator on
$L^{2}(\mathcal{X})$, with
$L+M_{0}\ge0$ for some constant $M_{0}\ge0$,
satisfying condition \eqref{eq1-assumption-0}
with $m_{1}=m_{2}=m>0$.
Moreover, there exists a weight function $w(x,y)$ such that
the resolvent $R(z)=(L+z)^{-1}$ satisfies, for all
$0\le j\le\lfloor n/2\rfloor+1$,
\begin{equation}\label{eq2-assumption-1}
  \|{\rm Ad}^{j}_{x}(R(M))\|_{2\to2}\le
  C(M-M_{0})^{-1-\f{j}{m}},
  \quad
  \forall
  M>M_{0}, \ 
  x\in \mathcal{X}.
\end{equation}

\begin{rem}[]\label{rem:}
  Note that when $m_{1}=m_{2}$ the following implication holds:
  \begin{equation*}
    \ \text{Assumption \textbf{(L$_{0}$)}}\ 
    \implies
    \ \text{Assumption \textbf{(L$_{1}$)} (with $w(x,y)=d(x,y)$)}\ 
  \end{equation*}
  (compare with Remark \ref{rem:LL}).
  This is proved in Propositions \ref{pro:weakerL}
  and \ref{pro:commR} below.
\end{rem}

Under this assumption we can prove:

\begin{thm}[]\label{the:third}
  Assume 
  $L$ satisfies Assumption \textbf{(L$_{1}$)}. Let $p\in [p_0,p_0']$  and let $s=n\Big|\f{1}{2}-\f{1}{p}\Big|$.
  Then we have
  \begin{equation*}
    \|e^{-itL}\varphi(\theta L)f\|_{p}\les (1+\theta|t|)^s\|f\|_p,
    \qquad t\in \mathbb{R},
  \end{equation*}
  and the estimate is uniform for $\varphi$ in bounded subsets
  of $\mathscr{S}(\mathbb{R})$ and $\theta$ in bounded subsets of
  $(0,+\infty)$. In the special case when
  $\kappa=n$ and $M_{0}=0$
  the estimate is uniform for all $\theta>0$.
\end{thm}

\begin{rem}[]
  Like for Theorems \ref{the:first} and \ref{the:second},
  the previous estimate is valid and uniform in the
  more general case of functions $\varphi$ varying in
  any bounded subset for the weighted Sobolev
  norm \eqref{eq:weightedsob0}.
\end{rem}

As an intermediate step in the proof of the previous
Theorems, we obtain uniform
$L^{p}$ estimates for operators of the form $\varphi(L)$
which are of independent interest, see
Theorem \ref{the:phiLscal}. (This result can be
recovered from the statement of
Theorem \ref{the:third} choosing $t=0$).

Our results are based on a commutator argument and
a reduction to amalgam spaces, following the methods of 
Jensen--Nakamura \cite{JensenNakamura95-a}.
The adaptation of the argument from \cite{JensenNakamura95-a}
to a multi--scale setting was introduced in
\cite{DAnconaNicola15-a} and was inspired by
the ideas of \cite{Tao99-a}. Moreover, our approach can be adapted to study the $L^p$-boundedness for Schr\"odinger group on an open subset of the space of homoegeneous type $\mathcal{X}$. 

We finally consider a self--adjoint operator $L$ on $L^{2}(\Omega)$, where $\Omega$ is an open subset of $\mathcal{X}$. This case
can not be reduced to the previous results since
$\Omega$ may not satisfy the doubling condition. 
However, if we assume that $L+M_{0}\ge0$ for some 
$M_{0}\ge0$ and the kernel $p_t(x,y)$ of heat semigroup $e^{-tL}$ satisfies the following estimate: $\exists C\ge0$, $m>1$ such that
\begin{equation}\label{Assumption-Subdomain}
|p_t(x,y)|\leq \f{Ce^{M_0t}}{\mu(B(x,t{^{1/m}}))}\exp\Big(-\f{d(x,y)^{m/(m-1)}}{ct^{1/(m-1)}}\Big)
\end{equation}
for all $t>0$ and $x,y\in \Omega$, then we can prove:

\begin{thm}[]\label{the:fourth}
	Let $L$ be a  nonnegative self--adjoint operator on $L^2(\Omega)$, where $\Omega$ is an open subset of $\mathcal{X}$. Assume that $L$ satisfies \eqref{Assumption-Subdomain}. Let $p\in [1,\vc]$  and let $s=n\Big|\f{1}{2}-\f{1}{p}\Big|$. Then we have
	\begin{equation*}
	\|e^{-itL}\varphi(\theta L)f\|_{L^p(\Omega)}\les (1+\theta|t|)^s\|f\|_{L^p(\Omega)},
	\qquad t\in \mathbb{R},
	\end{equation*}
	and the estimate is uniform for $\varphi$ in bounded subsets
	of $\mathscr{S}(\mathbb{R})$ and $\theta$ in bounded subsets of
  $(0,+\infty)$. In the special case
	$\kappa=n$ and $M_0=0$
	the estimate is uniform for all $\theta>0$.
\end{thm}

The proofs of the Theorems,
and some additional estimates, are given in the next section.
The third, and final, section of the paper
is devoted to an extensive list of applications: we consider
Laplace--Beltrami operators on Riemannian manifolds
with or without Gaussian heat kernel bounds;
the operator associated to the Sierpinski gasket;
H\"{o}rmander type operators generated by vector fields on
homogeneous groups;
Bessel operators;
Schr\"{o}dinger operators with potentials on manifolds;
euclidean Schr\"{o}dinger operators with singular potentials of
inverse square type; 
the sub-Laplacian on Heisenberg groups;
and Dirichlet Laplacian on open connected domains.
The list is not exhaustive and is intended to
show the variety of possible
applications and the generality of Assumption \textbf{(L)}.

\section{Proof of the theorems}\label{sec:pre}

With the notation $V(x,r)=\mu(B(x,r))$,
the doubling property (\ref{doublingpro}) 
implies the existence of $C>0$ and $n>0$ such that
\begin{equation}\label{doublingpro1}
  V(x,\lambda r)\leq C\lambda^n V(x,r),
  \qquad
  \forall\lambda>0,\ x\in \mathcal{X},
\end{equation}
and
\begin{equation}\label{doublingpro2}
  V(x,r)\leq C\Big(1+\frac{d(x,y)}{r}\Big)^n V(y,r),
  \qquad
  \forall r>0,\ x,y\in \mathcal{X}.
\end{equation}
As a consequence of \eqref{doublingpro2}, 
we have $V(x,r)\simeq V(y,r)$ when $d(x,y)\leq r$.

We recall the fundamental covering lemma from \cite{C}:

\begin{lem}\label{Christ'slemma} 
  There exists a collection of open sets
  $\{Q_\tau^k\subset \mathcal{X}: k\in
  \mathbb{Z}, \tau\in I_k\}$, where $I_k$ denotes certain (possibly
  finite) index sets depending on $k$, and constants $\rho\in (0,1)$
  $c_0\in (0,1]$ and $C_0,C_{1}\in (0,\vc)$ such that
  \begin{enumerate}[(i)]  
    \item $\mu(\mathcal{X}\backslash \cup_\tau Q_\tau^k)=0$ 
    for all $k\in \mathbb{Z}$;
    \item if $\ell\geq k$ and $\tau\in I_\ell, \beta\in I_k$, 
    then either $Q_\tau^\ell \subset Q_\beta^k$ or $Q_\tau^\ell \cap
    Q_\beta^k=\emptyset$;
    \item for $k\in \mathbb{Z}$, $\tau\in I_k$ 
    and each $\ell<k$, there exists a unique $\tau'\in I_\ell$
    such that $Q_\tau^k\subset Q_{\tau'}^\ell$;
    \item the diameters of the sets satisfy
    ${\rm diam}\,(Q_\tau^k)\leq C_1 \rho^k$;
    \item for $k\in \mathbb{Z}$, $\tau\in I_k$ 
    there exists $x_{Q_\tau^k}\in X$ such that
    $$B(x_{Q_\tau^k}, c_0\rho^k)\subset 
    Q_\tau^k\subset B(x_{Q_\tau^k}, C_0\rho^k).
    $$
  \end{enumerate}
\end{lem}

\begin{rem}\label{rem1}
  (a)
  The constants $\rho, c_0$ and $C_0$ are inessential for our 
  purposes, thus, without loss of generality, we may assume that
  $\rho=a_0=1/2$ and $C_0=1$. We then fix a collection of open sets
  in Lemma \ref{Christ'slemma} and denote this collection by 
  $\mathcal{D}$. We call these open sets the \emph{dyadic cubes} in 
  $\mathcal{X}$ and $x_{Q_\tau^k}$ the \emph{center} of the cube 
  $Q_\tau^k$. We also write 
  $\mathcal{D}_\nu:=\{Q_\tau^{\nu}: \tau\in I_{\nu}\}$ for each 
  $\nu\in \mathbb{Z}$. We have then 
  $\ell_Q:={\rm diam}\; Q \sim 2^{-\nu}$
  for all $Q\in \mathcal{D}_\nu$.
  
  (b)
  From the doubling property \eqref{doublingpro1}, there exists a constant $C$ such that for any $x\in \mathcal{X}$ and $k\in \mathbb{N}$ there are at most $C2^{kn}$ dyadic cubes in $\mathcal{D}_0$ which cover the ball $B(x,2^k)$.
\end{rem}

\subsection{Amalgam spaces} 

For $1\leq p, q\leq \vc$ and $\nu\in \mathbb{Z}$, we define the 
space $X^{p,q}_\nu$ as the vector space of all measurable functions 
$f:\mathcal{X}\to \mathbb{C}$ such that the following norm is 
finite:
\begin{equation}\label{eq:defamalgam}
  \|f\|_{X^{p,q}_\nu}:=\Big(\sum_{Q\in \mathcal{D}_\nu}
  \|f\|_{L^q(Q)}^p\Big)^{1/p}
\end{equation}
with the usual modification when $p=\vc$. 
We also write $X^{p,q}=X_0^{p,q}$.

The following embedding holds:

\begin{prop}
  \label{prop-embedding}
  For $1\leq p\leq q\leq \vc$ and $\nu\in \mathbb{Z}$ we have
  $$
  \|f\|_{X^{p,q}}\le C
   \Big(1+2^{-\nu n(\f{1}{p}-\f{1}{q})}\Big)
  \|f\|_{X^{p,q}_\nu}
  $$
  where $C$ depends only on the constant $c_{1}$
  in the doubling property \eqref{doublingpro}.
\end{prop}

\begin{proof}
  The proof of this proposition is elementary and we leave it to the
  reader.
\end{proof}

Recall that ${\rm Ad}^{j}_{x}(T)$ denotes the $j$-th order
commutator of an operator $T$ with the weight function
$w_{x}(\cdot)=w(x,\cdot)$, $w:\mathcal{X}\times \mathcal{X}\to 
  \mathbb{R}^{\ell}$, satisfying \eqref{eq:weight}.

\begin{thm} \label{thm-commutator}
  Let $T$ be a bounded operator on $L^2(\mathcal{X})$. 
  Assume that ${\rm Ad}^k_{z}$ can be extended to be a bounded operator on $L^2(X)$ and there exists some constant $B_0\geq 1$ so that 
  $$
  \|{\rm Ad}^k_{z}(T)\|_{2\to 2}\leq B_0^k
  $$
   for all $0\leq k\leq \lfloor n/2\rfloor+1$ 
  and all $z\in \mathcal{X}$.

  Then for $1\leq p\leq 2$ we have
  $$
  \|T\|_{X^{p,2}\to X^{p,2}}\leq C B_0^{n(1/p-1/2)}
  $$
  where $C$ is a constant depending only on $n$, $\|T\|_{2\to 2}$
  and $K_{0}$ (from \eqref{eq:weight}).
\end{thm}

\begin{proof}
  We first note that the $L^2$-boundedness of $T$ implies 
  $$
  \|T\|_{X^{2,2}_{\nu_{0}}\to X^{2,2}_{\nu_{0}}}\leq C.
  $$
  Hence, by interpolation it suffices to prove that 
  $$
    \|T\|_{X^{1,2}\to X^{1,2}}\leq C B_0^{n/2}.
  $$
    To prove this, let $w=(w_{1},\dots,w_{\ell})$
    be the weight function and recall that
    ${\rm Ad}_{j,z}^{k}(T)$ denotes the commutator of order $k$ with
    multiplication by $w_{j,z}:=w_{j}(z,\cdot)$.
    We use a combinatorial identity from 
    \cite[Lemma 3.1]{JensenNakamura95-a} and we write
    \begin{equation}\label{JN'sFormula}
      w_{j,z}^m T = \sum_{k=0}^m 
      c_{m,k} {\rm Ad}_{j,z}^k(T) w_{j,z}^{m-k},
      \qquad
      j=1,\dots,\ell
    \end{equation}
    where $c_{m,k}$ are appropriate constants.
    Denote also by $d_{z}$ the multiplication operator by
    $d(z,\cdot)$.
    Then, we have for every $x_Q$ with 
    $ Q\in \mathcal{D}_0$ and $N,m \in \mathbb{N}$ with
    $ 0\leq m\leq N\leq \lfloor n/2\rfloor +1$,
    $$
    \begin{aligned}
      \||w_{j,x_{Q}}|^m T[1+d_{x_{Q}}]^{-N}\|_{2\to 2}&
      \leq \sum_{k=0}^m c_{m,k} \|{\rm Ad}_{j,x_Q}^k(T)\|_{2\to 2} 
      \|w_{j,x_{Q}}^{m-k}[1+d_{x_{Q}}]^{-N}\|_{2\to 2}\\
      &\leq CB_0^m
    \end{aligned}
    $$
    since $|w_{j}|$ is dominated by $d$.
    Summing over $j=1,\dots,\ell$ and recalling
    \eqref{eq:weight} we obtain,
    for $0\leq N\leq \lfloor n/2\rfloor +1$,
    \begin{equation}\label{eq1.1}
    \left\|[1+d_{x_{Q}}]^{N} T [1+d_{x_{Q}}]^{-N}\right\|_{2\to 2}\leq CB_0^N.
    \end{equation}
    This implies
    \begin{equation}\label{eq1.2bis}
    \left\|d_{x_{Q}}^{N} T 1_Q\right\|_{2\to 2}\leq CB_0^N, \ \ \forall Q\in \mathcal{D}_0.
    \end{equation}
    Let $f\in \mathscr{D}(w)$. For each cube $Q$, write $f_Q=f1_Q$. 
    Then we have
    \begin{equation*}
    \begin{aligned}
      \|Tf\|_{X^{1,2}}=\sum_{Q'\in \mathcal{D}_0}\|1_{Q'}Tf\|_2\leq \sum_{Q\in \mathcal{D}_0}\sum_{Q'\in \mathcal{D}_0}\|1_{Q'}Tf_Q\|_2.
    \end{aligned}
    \end{equation*}
    Let now $\alpha$ be a constant which will be precised later;
    for the moment we assume only 
    $\alpha\ge 2\ \sup_{Q\in \mathcal{D}_{0}}{\rm diam}\ Q$.
    For each $Q\in \mathcal{D}_0$ we can write
    \begin{equation}\label{eq1.4}
      \sum_{Q'\in \mathcal{D}_0}\|1_{Q'}Tf_Q\|_2=I+II
    \end{equation}
    where
    \begin{equation*}
      I=\sum_{Q':d(x_Q,x_{Q'})>\alpha}
      d(x_Q,x_{Q'})^{-N}d(x_Q,x_{Q'})^N\|1_{Q'}Tf_Q\|_2,
    \end{equation*}
    \begin{equation*}
      II=\sum_{Q':d(x_Q,x_{Q'})\leq \alpha}\|1_{Q'}Tf_Q\|_2.
    \end{equation*}
    On the other hand, by Remark \ref{rem1} we get 
    \begin{equation}\label{eq1.3}
      \sharp\{Q'\in \mathcal{D}_0: 
      d(x_Q,x_{Q'})\leq \alpha\}\les \alpha^{n}.
    \end{equation}
    This, in combination with H\"older's inequality, implies that
    $$
    \begin{aligned}
    II&\leq \Big(\sum_{Q':d(x_Q,x_{Q'})\leq \alpha}1\Big)^{1/2}\Big(\sum_{Q':d(x_Q,x_{Q'})\leq \alpha}\|1_{Q'}Tf_Q\|^2_2\Big)^{1/2}\\
    &\les \alpha^{n/2}\|Tf_Q\|_2\\
    &\les \alpha^{n/2}\|T\|_{2\to 2}\|f_Q\|_2.
    \end{aligned}
    $$
    On the other hamd, for $x'\in Q'$ we have
    $d(x_{Q},x')\ge d(x_{Q},x_{Q'})-{\rm diam}(Q')$,
    thus if $d(x_Q,x_{Q'})\geq \alpha$ we have
    $d(x_{Q},x') \gtrsim d(x_Q,x_{Q'})$ by the assumption on $\alpha$.
    Then we can write
    $$
    \begin{aligned}
    I&\leq \Big(\sum_{Q':d(x_Q,x_{Q'})\geq \alpha}d(x_Q,x_{Q'})^{-2N}\Big)^{1/2}\Big(\sum_{Q':d(x_Q,x_{Q'})\geq \alpha}d(x_Q,x_{Q'})^{2N}\|1_{Q'}Tf_Q\|^2_2\Big)^{1/2}\\
  &\lesssim \Big(\sum_{Q':d(x_Q,x_{Q'})\geq \alpha}d(x_Q,x_{Q'})^{-2N}\Big)^{1/2}\|d(\cdot,x_{Q})^{N}Tf_Q\|_2
    \end{aligned}
    $$
    which along with \eqref{eq1.2bis} and \eqref{eq1.3} yields
    $$
    I\les \alpha^{-N+n/2} B_0^{N}\|f_Q\|_{2}
    $$
    provided that $N>n/2$.
  
  Inserting the estimates of $I$ and $II$ into \eqref{eq1.4} and taking 
  $\alpha=CB_{0}$ for a suitable $C$ (depending only
  on $\sup_{Q\in \mathcal{D}_{0}}{\rm diam}\ Q$)
  we obtain
  $$
  \sum_{Q'\in \mathcal{D}_0}\|1_{Q'}Tf_Q\|_2\les  
  (1+\|T\|_{2\to 2})\cdot B_0^{\f{n}{2}}\|f_Q\|_2.
  $$
  Therefore,
  $$
  \|Tf\|_{X^{1,2}}\les 
  (1+\|T\|_{2\to 2})\cdot B_0^{\f{n}{2}}\|f\|_{X^{1,2}}, \ \ \ f\in \mathscr{D}(w).
  $$
  Since $X^{1,2}\hookrightarrow L^2(X)$, $\mathscr{D}(w)$ is dense in $X^{1,2}$. It follows
  $$
  \|Tf\|_{X^{1,2}}\les 
  (1+\|T\|_{2\to 2})\cdot B_0^{\f{n}{2}}\|f\|_{X^{1,2}}, \ \ \ f\in X^{1,2}.
  $$
  On the other hand, since $T$ is bounded on $L^2$, we have
  $$
  \|Tf\|_{X^{2,2}}\les \|T\|_{2\to 2}\|f\|_{X^{2,2}}.
  $$
  Interpolating between the two estimates we get the claim.
\end{proof}

We conclude this section by proving that assumption 
\eqref{eq2-assumption} implies \eqref{eq2-assumption-0}
and \eqref{eq2-assumption-1}, as stated in the Introduction.

\begin{prop}[]\label{pro:weakerL}
  Let the weight function be $w(x,y)=d(x,y)$.
  Assume that $L$ is a self--adjoint operator in $L^2(\mathcal{X})$ 
  with $L+M_{0}\ge0$ for some $M_{0}\in \mathbb{R}$,
  satisfying the following condition: there exist 
  $p_0\in [1,2)$, $m_1, m_2>0$ and $C\ge0$ such that for all 
  $t>0$ and $\nu\in\mathbb{Z}$ with either 
  $2^{-\nu}\leq t^{1/m_1}<2^{-\nu+1}, 0<t<1$ or $2^{-\nu}\leq t^{1/m_2}<2^{-\nu+1}, t\geq 1$ we have
  \begin{equation}\label{eq2-assumptionbis}
  \sup_{Q'\in \mathcal{D}_\nu}\sum_{Q\in \mathcal{D}_\nu}(1+2^\nu {\rm dist}(Q,Q'))^{N}\|1_Q e^{-tL}1_{Q'}\|_{2\to 2}
  \leq C e^{M_{0}t}, \ \ \ \ 
  N=\lfloor n/2\rfloor+1.
  \end{equation}
  Then there exist $C_{1}\ge0$ such that
  for all $t$ and $\nu$ as above we have
  \begin{equation}\label{eq2-assumption-0bis}
    \|{\rm Ad}^{k}_{x}(e^{-tL})\|_{2\to2}\le
    C_{1}e^{M_{0}t}2^{-k\nu},
    \qquad
    0\le k\le\lfloor n/2\rfloor+1,\qquad x\in \mathcal{X}.
  \end{equation}
\end{prop}

\begin{proof}
  By considering the nonnegative
  operator $\widetilde{L}=L+M_{0}$ instead of $L$,
  we see that we can assume $M_{0}=0$.
  If $p_{t}(x,y)$ is the kernel of the heat semigroup $e^{-tL}$
  we obtain the representation
  \begin{equation*}
    {\rm Ad}^{k}_{z}(e^{-tL})f(x)=
    \int_{\mathcal{X}}
    (d_{z}(x)-d_{z}(y))^{k}p_{t}(x,y)
    f(y)d\mu(y)
  \end{equation*}
  and our goal is  to prove that the operator
  \begin{equation*}
    Af(x)=
    2^{\nu k}\int_{\mathcal{X}}
    (d_{z}(x)-d_{z}(y))^{k}p_{t}(x,y)
    f(y)d\mu(y)
  \end{equation*}
  for $0<t<1$ and $2^{-\nu}\le t^{1/m_{1}}<2^{-\nu+1}$,
  or 
  for $t\ge1$ and $2^{-\nu}\le t^{1/m_{2}}<2^{-\nu+1}$,
  satisfies $\|A\|_{2\to2}\le C$ with $C$ independent of $\nu$.

  We shall now prove the estimate
  \begin{equation}\label{eq:partialA}
    \sup_{Q'\in \mathcal{D}_{\nu}}\sum_{Q\in \mathcal{D}_{\nu}}
    \|1_{Q}A1_{Q'}\|_{2\to2}\le C_{1}
  \end{equation}
  with constants independent of $\nu$. This
  implies the dual estimate
  \begin{equation*}
    \sup_{Q'\in \mathcal{D}_{\nu}}\sum_{Q\in \mathcal{D}_{\nu}}
    \|1_{Q'}A1_{Q}\|_{2\to2}\le  C_{1}
  \end{equation*}
  and by the Schur test for sequences the two estimates together imply that
  $A$ is bounded on $X^{p,2}_{\nu}$ with norm not larger than
  $ C_{1}$,
  for all $p\in[1,\infty]$ and all $\nu\in \mathbb{Z}$.
  Since $L^{2}=X^{2,2}_{\nu}$ for all $\nu\in \mathbb{Z}$, 
  this concludes the proof.

  It remains to prove \eqref{eq:partialA}. We write
  the kernel of $1_{Q}A1_{Q'}$ as
  \begin{equation*}
    1_{Q}A1_{Q'}(x,y)=2^{\nu k}(d_{z}(x)-d_{z}(y))^{k}
    1_{Q}(x)p_{t}(x,y)1_{Q'}(y)
  \end{equation*}
  and we use the estimate
  \begin{equation*}
    |d_{z}(x)-d_{z}(y)|\le
    d(x_{Q},x_{Q'})+d(x,x_{Q})+d(y,x_{Q'}),\ x\in Q,\ y\in Q'
  \end{equation*}
  where $Q \subset B(x_{Q},2^{-\nu})$ and $Q' \subset B(x_{Q'},2^{-\nu})$
  according to Remark \ref{rem1}. We now expand
  \begin{equation*}
    \textstyle
    |1_{Q}A1_{Q'}(x,y)|\le
    \sum\limits_{\alpha+\beta+\gamma=k}
    \frac{k!}{\alpha!\beta!\gamma!}
    (2^{\nu}d(x_{Q},x_{Q'}))^{\alpha}
    (2^{\nu}d(x,x_{Q}))^{\beta}
    1_{Q}(x)p_{t}(x,y)1_{Q'}(y)
    (2^{\nu}d(y,x_{Q'}))^{\gamma}.
  \end{equation*}
  We have trivially
  \begin{equation*}
    \|(2^{\nu}d(x,x_{Q}))^{\beta}1_{Q}\|_{2\to2}\le C
    \qquad
    \|(2^{\nu}d(y,x_{Q'}))^{\gamma}1_{Q}\|_{2\to2}\le C,
  \end{equation*}
  and recalling assumption \eqref{eq2-assumptionbis} we see that
  the proof is concluded.
\end{proof}

From condition \eqref{eq2-assumption-0bis} it is fairly easy
to deduce \eqref{eq2-assumption-0}, thus concluding the
proof of the implication 
\eqref{eq2-assumption} $\Rightarrow$ \eqref{eq2-assumption-0},
\eqref{eq2-assumption-1}.

\begin{prop}\label{pro:commR}
  Let the weight function be $w(x,y)=d(x,y)$.
  Assume $L$ satisfies \eqref{eq2-assumption-0bis}
  and $L+M_{0}\ge0$. Then for all $M>M_{0}$ we have,
  for all $z\in \mathcal{X}$ and
  $0\le k\le\lfloor n/2\rfloor+1$,
  \begin{equation}\label{eq:L2AdR}
   \|{\rm Ad}^k_{z}((L+M)^{-1})\|_{2\to 2}
   \lesssim
   (M-M_{0})^{-1-\f{k}{m_{1}}}+
   (M-M_{0})^{-1-\f{k}{m_{2}}}
  \end{equation}
  with a constant independent of $z,M$. 
\end{prop}

\begin{proof}
  By spectral calculus we can represent $R=(M+L)^{-1}$
  in the form
  \begin{equation*}
    \textstyle
    R=(M+L)^{-1}=
    \int_0^\vc  e^{-Mt}e^{-tL}dt
  \end{equation*}
  which implies
  \begin{equation*}
    \textstyle
    {\rm Ad}^{k}_{z}(R)=
    \int_{0}^{\infty}e^{-Mt} {\rm Ad}^{k}_{z}(e^{-tL})dt.
  \end{equation*}
  By assumption \eqref{eq2-assumption-0bis},
  since $2^{-\nu}\simeq t^{^{1/m_{2}}}$ for $t<1$
  and $2^{-\nu}\simeq t^{^{1/m_{1}}}$ for $t>1$,
  we obtain
  \begin{equation*}
    \textstyle
    \|{\rm Ad}^{k}_{z}(R)\|_{2\to2}
    \lesssim
    \int_{0}^{1}e^{(M_{0}-M)t}t^{k/m_{2}}dt
    +
    \int_{1}^{+\infty}e^{(M_{0}-M)t}t^{k/m_{1}}dt
  \end{equation*}
  and the claim follows easily.
\end{proof}




\subsection{Estimates for the heat semigroup}\label{sub:the_hea_ker}

The following result gives an estimate for the semigroups 
$e^{-tL}$ on almagam spaces which plays an important role in 
the sequel.

\begin{prop}
\label{prop-boundednessOfsemigroups}
For every $t>0$ we have
$$
\|e^{-tL}f\|_{L^{p_{0}}\to X^{p_{0},2}}\
\le C
e^{M_{0}t}
\bigl(  
t^{-\f{n}{m_1}(\f{1}{p_{0}}-\f{1}{2})}+
t^{\f{n-\kappa}{m_2}(\f{1}{p_{0}}-\f{1}{2})}
\bigr)
$$
where $C$ depends only on the constants $C$ in assumption
\eqref{eq1-assumption-0} and $c_{1}$ in
\eqref{doublingpro}.
\end{prop}

\begin{proof}
  By redefining $\widetilde{L}=L+M_{0}$, we see that it is
  not restrictive to assume $M_{0}=0$.
  Now fix $\nu\in\mathbb{Z}$ and $t>0$ such that
  either $2^{-\nu}\leq t^{1/m_1}<2^{-\nu+1}, 0<t<1$ 
  or $2^{-\nu}\leq t^{1/m_2}<2^{-\nu+1}, t\geq 1$.
  By assumption \eqref{eq1-assumption-0}, using duality we have
  \begin{equation*}
    \|e^{-tL}\|_{X_{\nu}^{\infty,p_{0}}\to X_{\nu}^{\infty,2}}\le
    C \bigl(
    2^{\nu \kappa(\f{1}{p_0}-\f{1}{2})}
    +2^{\nu n(\f{1}{p_0}-\f{1}{2})}\bigr)
  \end{equation*}
  and interpolating with \eqref{eq1-assumption-0}
  we have, for all $1\le p\le \infty$,
  \begin{equation*}
    \|e^{-tL}\|_{X_{\nu}^{p,p_{0}}\to X_{\nu}^{p,2}}\le
    C \bigl(
    2^{\nu \kappa(\f{1}{p_0}-\f{1}{2})}
    +2^{\nu n(\f{1}{p_0}-\f{1}{2})}\bigr)
  \end{equation*}
  We choose $p=p_{0}$ and notice that
  $X_{\nu}^{p_{0},p_{0}}=L^{p_{0}}$; thus we have proved
  \begin{equation*}
    \|e^{-tL}\|_{L^{p_{0}}\to X_{\nu}^{p_{0},2}}\le
    C \bigl(
    2^{\nu \kappa(\f{1}{p_0}-\f{1}{2})}
    +2^{\nu n(\f{1}{p_0}-\f{1}{2})}\bigr)
  \end{equation*}
  By the embedding in Proposition \ref{prop-embedding}
  this implies
  \begin{equation*}
  \begin{split}
    \|e^{-tL}\|_{L^{p_{0}}\to X^{p_{0},2}}\le
    &
    C \bigl(
    2^{\nu \kappa(\f{1}{p_0}-\f{1}{2})}
    +2^{\nu n(\f{1}{p_0}-\f{1}{2})}\bigr)
    \bigl(1+2^{-\nu n(\f{1}{p_0}-\f{1}{2})}\bigr)
    \\
    &
    \simeq
    2^{\nu (\kappa-n)(\f{1}{p_0}-\f{1}{2})}
    +2^{\nu n(\f{1}{p_0}-\f{1}{2})}
    \end{split}
  \end{equation*}
  and recalling the conditions on $t$, we obtain the claim.
\end{proof}

As a consequence we obtain the following result.

\begin{prop} \label{prop-boundednessofRgamma}
  Let  $M>M_{0}$ and
  $\gamma=\f{n}{m_1}(\f{1}{p_0}-\f{1}{2})+\epsilon$,
  with $\epsilon>0$. Then
  $$
  \textstyle
  \|(M+L)^{-\gamma}f\|_{X^{p_0,2}}\le
  C
  \Bigl(
    \epsilon ^{-1}
    +
    (M-M_{0})
    ^{\gamma+\f{n-\kappa}{m_2}(\f{1}{p_{0}}-\f{1}{2})}
  \Bigr)
   \|f\|_{p_0}
  $$
  where $C$ depends only on the constants $C$ in assumption
  \eqref{eq1-assumption-0} and $c_{1}$ in
  \eqref{doublingpro}.
\end{prop}

\begin{proof}
  It is sufficient to apply Minkowski's inequality and
  Proposition \ref{prop-boundednessOfsemigroups} to 
  the standard representation
  \begin{equation}\label{eq:resolvheat}
    (M+L)^{-\gamma}=
    \f{1}{\Gamma(\gamma)}\int_0^\vc t^\gamma e^{-Mt}e^{-tL}\f{dt}{t}.
  \end{equation}
\end{proof}

\subsection{Estimate of \texorpdfstring{$\varphi(L)$}{}}
\label{sub:est_of_tex}

We shall now prove that if $\varphi$ is in a suitable
weighted Sobolev space then
$\varphi(L)$ is bounded on $L^{p}$.
The proof will be achieved through a series of Lemmas, some of which
are of independent interest.

In the following, $L$ is an operator satisfying Assumption 
\textbf{(L)},
and we can take $R$ as the resolvent operator
\begin{equation*}
  R=(M_{1}+L)^{-1}
\end{equation*}
with $M_{1}>M_{0}$ as in \textbf{(L)}.

\begin{lem}[]\label{lem:eixiR}
  We have the estimate
  \begin{equation*}
    \|e^{-i \xi R}f\|_{X^{p_{0},2}}\le
    c(n)
    C(1+|\xi|)^{n(1/p_{0}-1/2)}\|f\|_{_{X^{p_{0},2}}},
    \qquad
    \xi\in \mathbb{R}
  \end{equation*}
  where $C$ is the constant in assumption \eqref{eq2-assumption-0}
  and $c(n)$ depends only on $n$.
\end{lem}

\begin{proof}
  From 
  \begin{equation*}
    e^{-i\xi R}w_{z}(\cdot)e^{i\xi R}-w_{z}(\cdot)=
    \int_{0}^{\xi}\partial_{s}(e^{-is R}w_{z}(\cdot)e^{is R})ds
  \end{equation*}
  we obtain the formula
  \begin{equation*}
    {\rm Ad}_{z}(e^{-i \xi R})=
    -i\int_{0}^{\xi}
    e^{-is R}{\rm Ad}_{z}(R)e^{-i(\xi-s)R}ds
  \end{equation*}
  and by \eqref{eq2-assumption-0} we get
  \begin{equation*}
    \|{\rm Ad}^{1}_{z}(e^{-i\xi R})\|_{2\to 2}\leq C|\xi|,
  \end{equation*}
  Using repeatedly this identity and
  proceeding by induction we obtain
  \begin{equation*}
    \|{\rm Ad}^{k}_{z}(e^{-i\xi R})\|_{2\to 2}\leq C(1+|\xi|)^{k},
    \qquad
    k=0,\dots,\lfloor n/2 \rfloor+1
  \end{equation*}
  uniformly in $z\in \mathcal{X}$,
  and by Theorem \ref{thm-commutator} we obtain the claim.
\end{proof}

\begin{lem}[]\label{lem:psiR}
  For any sufficiently smooth function $\psi$ on $\mathbb{R}$
  we have the estimate
  \begin{equation}\label{eq:estpsiR}
    \|\psi(R)f\|_{X^{p_{0},2}}\le
    c(n)
    C\|(1+|\xi|)^{n(1/p_{0}-1/2)}\widehat{\psi}(\xi)\|_{L^{1}} 
    \|f\|_{X^{p_{0},2}}
  \end{equation}
  with $c(n),C$ as in Lemma \ref{lem:eixiR}.
\end{lem}

\begin{proof}
  It is sufficient to use the identity
  \begin{equation*}
    \psi(R)=(2\pi)^{-1}\int e^{i \xi R}\widehat{\psi}(\xi)d\xi
  \end{equation*}
  and apply the previous result.
\end{proof}

We introduce a seminorm for functions
$\psi:\mathbb{R}\to \mathbb{C}$, depending on
the constant $M_{1}\ge0$ and on the integer $N\ge0$:
\begin{equation*}
  \textstyle
  \bbar \psi\bbar_{N}:=
  \|\psi\| _{L^{2}(-M_{1},\infty)}
  +
  \sum_{j=0}^{N}
  \|(\lambda+M_{1})^{j+N}\partial^{j}\psi(\lambda)\|
  _{L^{2}(-M_{1},\infty)}.
\end{equation*}

\begin{lem}[]\label{lem:phiL}
  Let $N=\lfloor n/p_{0}\rfloor+1$
  and $\psi:\mathbb{R}\to \mathbb{C}$. Then we have
  \begin{equation}\label{eq:estphiL}
    \|(L+M_{1})^{2}\psi(L)f\|_{X^{p_{0},2}}\le 
    c(n)C \bbar \psi\bbar_{N}\cdot
    \|f\|_{X^{p_{0},2}}
  \end{equation}
  with $c(n),C$ as in Lemma \ref{lem:eixiR}.
\end{lem}

\begin{proof}
  Define $\rho(\xi):=0$ for $\xi\le0$, and
  \begin{equation*}
    \textstyle
    \rho(\xi):=\xi^{-2}\cdot\psi(\frac1\xi-M_{1})
    \quad\text{for}\quad 
    \xi>0
  \end{equation*}
  and note that 
  $(\lambda+M_{1})^{2}\psi(\lambda)=\rho((M_{1}+\lambda)^{-1})$ for
  $\lambda$ in the spectrum of $L$, so that 
  $(L+M_{1})^{2}\psi(L)=\rho(R)$. By the previous result we get
  \begin{equation*}
    \|(L+M_{1})^{2}\psi(L)f\|_{X^{p_{0},2}}
    \le c(n)C
    \|(1+|\xi|)^{n(1/p_{0}-1/2)}\widehat{\rho}(\xi)\|_{L^{1}}
    \|f\|_{X^{p_{0},2}}.
  \end{equation*}
  It remains to estimate the norm of $\rho$.
  We proceed as follows:
  \begin{equation*}
    \|(1+|\xi|)^{n(1/p_{0}-1/2)}\widehat{\rho}(\xi)\|_{L^{1}}
    \lesssim
    \|(1+|\xi|)^{N}\widehat{\rho}\|_{L^{2}}=
    \|\rho\|_{H^{N}(\mathbb{R}^{+})}.
  \end{equation*}
  We note the elementary identity for $\xi>0$,
  $k\ge0$
  \begin{equation*}
    \textstyle
    \partial^{k}_{\xi}\rho(\xi)=
    \sum_{j=0}^{k}
    c_{j,k} \cdot
    \partial^{j}\psi(\frac1\xi-M_{1})
    \cdot
    \xi^{-(j+k+2)}
  \end{equation*}
  (for suitable constants $c_{j,k}$). This gives
  \begin{equation*}
    \textstyle
    \|\partial^{k}_{\xi}\rho\|_{L^{2}(0,\infty)}
    \le
    c(n)
    \sum_{j=0}^{k}
    \|(\lambda+M_{1})^{j+k}
      \partial^{j}\psi(\lambda)\|_{L^{2}(-M_{1},\infty)}.
  \end{equation*}
  Using the last estimate for $k=0$ and $k=N$ we obtain
  \begin{equation*}
    \textstyle
    \|\rho\|_{H^{N}(\mathbb{R}^{+})}
    \le
    c(n)\|\psi\| _{L^{2}(-M_{1},\infty)}
    +
    c(n)
    \sum_{j=0}^{N}
    \|(\lambda+M_{1})^{j+N}\partial^{j}\psi(\lambda)\|
    _{L^{2}(-M_{1},\infty)}
  \end{equation*}
  and we obtain the claim.
\end{proof}

\begin{lem}[]\label{lem:phiL2}
  Let $N= \lfloor n/p_{0}\rfloor+1$, $\beta\ge0$ with
  $\beta+2>\frac{n}{m_{1}}(\frac{1}{p_{0}}-\frac12)$ and
  $\psi:\mathbb{R}\to \mathbb{C}$.
  Then for $p\in[p_{0},p_{0}']$ we have the estimate
  \begin{equation}\label{eq:estphiL2}
    \|\psi(L)f\|_{L^{p}}\les
    \|f\|_{L^{p}}.
  \end{equation}
  The norm of $\psi(L):L^{p}\to L^{p}$ can be estimated by
  \begin{equation}\label{eq:estconst}
    C(1+(M_{1}-M_{0})
    ^{\beta+2+\f{n-\kappa}{m_2}(\f{1}{p_{0}}-\f{1}{2})})
    \bbar (\lambda+M_{1})^{\beta}\psi(\lambda)\bbar_{N}
  \end{equation}
  where $C$ depends on $c_{1}$ in the doubling property
  \eqref{doublingpro}, on
  $\sup_{Q\in \mathcal{Q}_{0}}\mu(Q)$ and on
  the constants in Assumption \textbf{(L)},
  but is independent of $M_{0},M_{1}$,$\psi$.
\end{lem}

\begin{proof}
  We apply the previous Lemma to the function 
  $\widetilde{\psi}(\lambda)=(\lambda+M_{1})^{\beta+2}\psi(\lambda)$:
  \begin{equation*}
  \begin{split}
    \|\widetilde{\psi}(L)f\|_{X^{p_{0},2}} 
    &=
    \|(L+M_{1})^{2}\psi(L)(L+M_{1})^{\beta}f\|_{X^{p_{0},2}} 
    \\
    &\le
    c(n)C\bbar(\lambda+M_{1})^{\beta} \psi\bbar_{N}
    \|f\|_{X^{p_{0},2}}.
    \end{split}
\end{equation*}
  Since $\psi(L)=\widetilde{\psi}(L)R^{\beta+2}$, we can write,
  using Proposition \ref{prop-boundednessofRgamma},
  \begin{equation*}
    \|\psi(L)f\|_{X^{p_{0},2}}\le
    \|\widetilde{\psi}(L)\|_{X^{p_{0},2}\to X^{p_{0},2}}
    \|R^{\beta+2}f\|_{X^{p_{0},2}}\lesssim
    \|\widetilde{\psi}(L)\|_{X^{p_{0},2}\to X^{p_{0},2}}
    \|f\|_{L^{p_{0}}}
  \end{equation*}
  where the implicit constant has the form
  \begin{equation*}
    C(1+(M_{1}-M_{0})
    ^{\beta+2+\f{n-\kappa}{m_2}(\f{1}{p_{0}}-\f{1}{2})})
  \end{equation*}
  with $C$ depending on $c_{1}$ in the doubling property
  \eqref{doublingpro} and on
  the constants in Assumption \textbf{(L)},
  but independent of $M_{0},M_{1}$.
  Since  $X^{p_{0},2}$ is continuously embedded in
  $L^{p_{0}}$ with embedding norm 
  $\le\sup_{Q\in \mathcal{Q}_{0}}\mu(Q)^{\f{1}{p_{0}}-\f12}$,
  we have proved that $\psi(L)$
  is bounded on $L^{p_{0}}$ with the same norm.
  By duality and interpolation we condlude the proof.
\end{proof}

\begin{rem}[]\label{rem:constant}
  The dependence on $\psi$ of the norm of $\psi(L)$
  is particularly interesting. The quantity
  $\bbar (\lambda+M_{1})^{\beta}\psi(\lambda)\bbar_{N}$
  is uniformly bounded if $\psi$ varies in a bounded subset
  of $C^{\infty}_{c}(\mathbb{R})$ or of $\mathscr{S}(\mathbb{R})$,
  and $M_{1}$ is bounded.
  More generally, we can write
  \begin{equation*}
    \textstyle
    \bbar (\lambda+M_{1})^{\beta}\psi(\lambda)\bbar_{N}
    \lesssim
    \|\lambda^{\beta}\psi(\lambda-M_{1})\| _{L^{2}(\mathbb{R}^{+})}
    +
    \sum_{j=0}^{N}
    \|\lambda^{j+N+\beta}\partial^{j}\psi(\lambda-M_{1})(\lambda)\|
    _{L^{2}(\mathbb{R}^{+})}
  \end{equation*}
  and we see that the quantity is uniform for $\psi$ varying in
  any bounded subset of a suitable weighted Sobolev space,
  provided $M_{1}$ is bounded (which is always the case in our
  applications). For instance, we can take the weighted
  Sobolev space with norm
  \begin{equation}\label{eq:weightedsob}
    \sum_{j\le n+1}
    \|\bra{\lambda}^{2+n+j+n/m_{1}}\psi^{(j)}(\lambda)\|_{L^{2}}.
  \end{equation}
\end{rem}

\begin{thm}[]\label{the:phiLscal}
  Under Assumption \textbf{(L)} the following estimate holds:
  for all $p\in[p_{0},p_{0}']$,
  \begin{equation*}
    \|\varphi(L)f\|_{L^{p}}\le C\|f\|_{L^{p}}
  \end{equation*}
  and the estimate is uniform for $\varphi$ in bounded subsets
  of $\mathscr{S}(\mathbb{R})$
  (or, more generally, in bounded subsets for the norm
  \eqref{eq:weightedsob}).

  If the stronger Assumption \textbf{(L$_{1}$)} holds,
  then for all $\theta>0$ we have
  \begin{equation*}
    \|\varphi(\theta L)f\|_{L^{p}}\le C\|f\|_{L^{p}}
  \end{equation*}
  and the estimate is uniform for 
  $\varphi$ in bounded subsets of $\mathscr{S}(\mathbb{R})$
  (or, more generally, in bounded subsets for the norm
  \eqref{eq:weightedsob})
  and $\theta$ in bounded subsets of $(0,+\infty)$.
  If in addition we assume $\kappa=n$
  and $M_{0}=0$, then
  the estimate is uniform for all $\theta>0$.
\end{thm}

\begin{proof}
  The first claim is just a special case of the previous Lemma.
  Thus we assume that \textbf{(L$_{1}$)} holds and we focus on the
  second claim. 
  Clearly it is sufficient to prove the result
  for all $\theta>0$ of the form
  \begin{equation*}
    \theta=2^{-m\gamma }
    \quad\text{for some}\quad 
    \gamma\in \mathbb{Z}.
  \end{equation*}
  Thus we fix a $\theta=2^{-m\gamma }>0$ and
  define a new metric measure space 
  $(\overline{\mathcal{X}},\overline{d},\overline{\mu})$ 
  by multiplying $d$ and $\mu$ by fixed constants, as follows:
  \begin{equation*}
    \overline{\mathcal{X}}=\mathcal{X},
    \qquad
    \overline{d}=2^{\gamma}d,
    \qquad
    \overline{\mu}=2^{n \gamma}\mu.
  \end{equation*}
  Note the relation
  \begin{equation*}
    \|u\|_{L^{p}(\overline{\mathcal{X}},d \overline{\mu})}=
    2^{\f{n \gamma}{p}}\|u\|_{L^{p}(\mathcal{X},d\mu)}.
  \end{equation*}
  Writing
  \begin{equation*}
    \overline{\mathcal{D}}_{\nu}=
    \mathcal{D}_{\nu+\gamma}
  \end{equation*}
  we see that the $\overline{\mathcal{D}}_{\nu}$ form a
  collection of dyadic cubes for the space 
  $\overline{\mathcal{X}}$, and with respect to the new
  distance $\overline{d}$ we have
  ${\rm diam}\; Q \sim 2^{-\nu}$ for all 
  $Q\in \overline{\mathcal{D}}_{\nu}$.
  Then if we define the amalgam spaces
  $\overline{X}^{p,q}_{\nu}$ as in \eqref{eq:defamalgam}
  but with $\overline{\mathcal{D}}_{\nu}$ instead
  of $\mathcal{D}_{\nu}$ and with the $L^{q}(Q)$
  norms computed in the measure $\overline{\mu}$, we get
  \begin{equation*}
    \|f\|_{\overline{X}^{p,q}_{\nu}}=
    2^{\f{n \gamma}{q}}\|f\|_{X^{p,q}_{\nu+\gamma}}.
  \end{equation*}

  Next, we denote
  by $\overline{L}$ the operator $\theta L$, which
  is self--adjoint on $L^{2}(\overline{\mathcal{X}})$ and satisfies
  $\overline{L}+\overline{M_{0}}\ge0$ with 
  $\overline{M_{0}}=\theta M_{0}$.
  To prove the claim, it will be sufficient to prove that
  the operator $\overline{L}$ satisfies the conditions of
  Assumption \textbf{(L)},
  with constants independent of $\theta$ in the prescribed range.
  By the first part of the Theorem, the claim will follow.

  Fix a $t>0$ and $\nu\in \mathbb{Z}$ as in condition
  \eqref{eq1-assumption-0} with $m_{1}=m_{2}=m$, i.e.,
  \begin{equation*}
    2^{-\nu}\leq t^{1/m}<2^{-\nu+1}.
  \end{equation*}
  Consider the first term in \eqref{eq1-assumption-0}
  (the second one is handled in a similar way):
  \begin{equation*}
    \|e^{-t \overline{L}}\|
    _{\overline{X}_{\nu}^{1,p_{0}}\to \overline{X}_{\nu}^{1,2}}
    =
    \|e^{-(\theta t) L}\|
    _{X_{\nu+\gamma}^{1,p_{0}}\to X_{\nu+\gamma}^{1,2}}
    \cdot
    2^{n \gamma(\f{1}{2}-\f{1}{p_{0}})};
  \end{equation*}
  using assumption \eqref{eq1-assumption-0} 
  (with $m_{1}=m_{2}=m$),
  since $2^{-(\nu+\gamma)}\le (\theta t)^{1/m}
    =2^{-\gamma}t <2^{-(\nu+\gamma)+1}$
  we get
  \begin{equation*}
    \le
    Ce^{M_{0}(\theta t)}
    ((\theta t)^{\f{n}{m}}\wedge 
      (\theta t)^{\f{\kappa}{m}})^{\f{1}{2}-\f{1}{p_{0}}}
    \cdot
    2^{n \gamma(\f{1}{2}-\f{1}{p_{0}})}
  \end{equation*}
  \begin{equation*}
    =C
    e^{\overline{M}_{0} t}
    (t^{\f{n}{m}}\wedge 
      t^{\f{\kappa}{m}}
      \cdot 2^{(n-\kappa)\gamma} )^{\f{1}{2}-\f{1}{p_{0}}}.
  \end{equation*}
  Thus we see that the operator $\overline{L}$
  also satisfies condition \eqref{eq1-assumption-0}
  with $m_{1}=m_{2}=m$. Note that the estimate
  is uniform in $\gamma$ provided $\gamma\ge \gamma_{0}$
  for some fixed $\gamma_{0}$, or equivalently,
  provided $\theta$ is bounded from above;
  moreover, $\overline{M}_{0}$ is also uniformly bounded from 
  above. It is also clear that if $\kappa=n$ and $M_{0}=0$
  the condition is uniform for all $\gamma\in \mathbb{Z}$,
  i.e., for all $\theta>0$.

  It remains to check condition \eqref{eq2-assumption-0};
  we choose as weight function and the space of ``'test functions'
  \begin{equation*}
    \overline{w}(x,y)=2^{\gamma}w(x,y), \ \ \text{and} \quad \mathscr D(\overline{w})=\mathscr D(w).
  \end{equation*}
  Writing $\overline{\rm{Ad}}^{j}_{x}$ for the commutators with
  the new weight function $\overline{w}$, we have
  \begin{equation*}
    \overline{\rm{Ad}}^{j}_{x}((\overline{L}+M)^{-1})
    =
    2^{m \gamma}2^{j \gamma}
    {\rm Ad}^{j}_{x}((L+2^{m\gamma}M)^{-1})).
  \end{equation*}
  By \eqref{eq2-assumption-1} we have then
  \begin{equation*}
    \|\overline{\rm{Ad}}^{j}_{x}((\overline{L}+M)^{-1})\|_{2\to2}
    \le
    C
    2^{m \gamma}2^{j \gamma}
    (2^{m \gamma}M-M_{0})^{-1-\f{j}{m}}
  \end{equation*}
  provided $2^{m \gamma}M> M_{0}$. Now, if $\gamma\ge \gamma_{0}$
  is bounded from below, we can choose
  $M=M_{1}=2^{-m \gamma_{0}}(M_{0}+1)$ and we get
  \begin{equation*}
    \le
    C
    2^{m \gamma}2^{j \gamma}
    (2^{m(\gamma-\gamma_{0})})^{-1-\f{j}{m}}\le C'
  \end{equation*}
  for some constant independent of $\gamma$.
  Note that if $M_{0}=0$ we have
  \begin{equation*}
    \|\overline{\rm{Ad}}^{j}_{x}((\overline{L}+M)^{-1})\|_{2\to2}
    \le
    C
    2^{m \gamma}2^{j \gamma}
    (2^{m \gamma}M)^{-1-\f{j}{m}}
    =
    CM^{-1-\f{j}{m}}
  \end{equation*}
  for all $M>0$, thus we can pick simply $M_{1}=1$
  without restrictions on $\gamma\in \mathbb{Z}$.
  The proof is concluded.
\end{proof}

\subsection{Proof of Theorems \ref{the:second}, \ref{the:third} and \ref{the:fourth}}
\label{sub:pro_of_the}

We keep using the notation
\begin{equation*}
  {\rm Ad}_{x}(A)=[w_{x},A],
  \qquad
  {\rm Ad}^{k}_{x}(A)=[w_{x},{\rm Ad}^{k-1}_{x}(A)]
\end{equation*}
for a generic operator $A$ and a $\mathbb{R}^{\ell}$ valued
weight function $w_{x}(\cdot)=w(x,\cdot)$.

\begin{lem}[]\label{lem:basiccom}
  For any $k\ge1$ and $z\in \mathcal{X}$ the following identities hold:
  \begin{equation*}
  \begin{split}
    {\rm Ad}_{z}(R^{2k}e^{-itL})=&
    \\
    =\sum_{\alpha=0}^{k}
    R^{\alpha}&{\rm Ad}_{z}(R)e^{-itL}R^{2k-\alpha-1}+
    \sum_{\alpha=0}^{k}
    R^{2k-\alpha-1}e^{-itL}{\rm Ad}_{z}(R)R^{\alpha}+
    \\
    &+i\int_{0}^{t}
    e^{-isL}R^{k-1}{\rm Ad}_{z}(R)R^{k-1}e^{i(s-t)L}ds,
  \end{split}
  \end{equation*}
  \begin{equation*}
  \begin{split}
    {\rm Ad}_{z}(R^{2k+1}e^{-itL})=&
    \\
    =\sum_{\alpha=0}^{k}
    R^{\alpha}&{\rm Ad}_{z}(R)e^{-itL}R^{2k-\alpha}+
    \sum_{\alpha=0}^{k+1}
    R^{2k-\alpha}e^{-itL}{\rm Ad}_{z}(R)R^{\alpha}+
    \\
    &+i\int_{0}^{t}
    e^{-isL}R^{k-1}{\rm Ad}_{z}(R)R^{k}e^{i(s-t)L}ds.
  \end{split}
  \end{equation*}
\end{lem}

\begin{proof}
  The first identity is proved by induction on $k$. 
  
  It can be verified that
  \begin{equation*}
    {\rm Ad}_{z}(Re^{-i t L}R)=
    {\rm Ad}_{z}(R)e^{-i t L}R+
    e^{-i t L}R{\rm Ad}_{z}(R)+
    i\int_{0}^{t}
    e^{-is L}{\rm Ad}_{z}(R)e^{i(s-t)L}ds
  \end{equation*}
  which is the first formula for $k=1$. 
  To prove the step $k\to k+1$ we write
  \begin{equation*}
    {\rm Ad}_{z}(R \cdot (R^{2k}e^{-itL})\cdot R)=
    {\rm Ad}_{z}(R)e^{-itL}R^{2k+1}+
    R^{2k+1}e^{-itL}{\rm Ad}_{z}(R)+I
  \end{equation*}
  where
  \begin{equation*}
    I=R \cdot {\rm Ad}_{z}(R^{2k}e^{-itL})\cdot R
  \end{equation*}
  and using the inductive assumption for the case $k$ we easily obtain
  the claim. The second formula is deduced from the first one
  writing
  \begin{equation*}
    {\rm Ad}_{z}((R^{2k}e^{-itL})\cdot R)=
    {\rm Ad}_{z}(R^{2k}e^{-itL}) R+
    R^{2k}e^{-itL}{\rm Ad}_{z}(R).
  \end{equation*}
\end{proof}

\begin{lem}[]\label{lem:ordlcom}
  For $0\le\ell\le k$ and $1\le k\le \lfloor n/2 \rfloor+1$ we have
  \begin{equation*}
    \|{\rm Ad}_{z}^{\ell}(R^{2k}e^{-itL})\|_{2\to2}\le
    C(1+|t|)^{\ell}
  \end{equation*}
  with $C$ independent of $z\in \mathcal{X}$ and $t\in\mathbb{R}$.
\end{lem}

\begin{proof}
  We proceed by induction on $k=1,\dots,\lfloor n/2 \rfloor+1$.
  When $k=1$, recalling the formulas from the previous Lemma
  and assumption \eqref{eq2-assumption-0}, we obtain the claim
  immediately. 
  Assume now the result is true for a certain $k$ and let us
  prove it for $k+1$. If $\ell=1$ the estimate follows again from
  the first identity in the previous Lemma. If the estimate is true
  for some $\ell<k$, we prove it for $\ell+1$ writing
  \begin{equation*}
    {\rm Ad}^{\ell+1}_{z}(R^{2k}e^{-itL})=
    {\rm Ad}^{\ell}_{z}({\rm Ad}_{z}(R^{2k}e^{-itL})),
  \end{equation*}
  expanding the term
  ${\rm Ad}_{z}(R^{2k}e^{-itL})$ via the first identity of the previous
  Lemma, and distributing the adjoint via the formula
  \begin{equation*}
    {\rm Ad}^{\ell}_{z}(A_{1}\dots A_{n})=
    \sum_{j_{1}+ \dots +j_{n}=\ell}
    \frac{\ell!}{j_{1}!\dots j_{n}!}
    {\rm Ad}^{j_{1}}_{z}(A_{1})
    \dots
    {\rm Ad}^{j_{n}}_{z}(A_{n}).
  \end{equation*}
  It is easy to check that all the terms obtained are bounded operators
  on $L^{2}$, either using the inductive assumption or
  \eqref{eq2-assumption-0}. The proof is concluded.
\end{proof}

\begin{lem}[]\label{lem:boundexp}
  Let $k=\lfloor n/2 \rfloor+1$. Then we have the estimates
  \begin{equation*}
    \|R^{2k}e^{-itL}\|_{X^{p_{0},2}\to X^{p_{0},2}}
    \le
    C(1+|t|)^{n(1/p_{0}-1/2)}
  \end{equation*}
  and, for all $p\in[p_{0},p_{0}']$ and 
  $\beta>\frac{n}{m_{1}}(\frac1{p_{0}}-\frac12)$,
  \begin{equation*}
    \|R^{2k+\beta}e^{-itL}\|_{L^{p}\to L^{p}}
    \le
    C(1+|t|)^{n|1/p-1/2|}.
  \end{equation*}
\end{lem}

\begin{proof}
  The first result is a direct application of Lemma \ref{lem:ordlcom}
  and Theorem \ref{thm-commutator}. Moreover, by Proposition
  \ref{prop-boundednessofRgamma} we have
  \begin{equation*}
    \|R^{2k+\beta}e^{-itL}\|_{X^{p_{0},2}\to L^{p_{0}}}
    \lesssim
    \|R^{2k}e^{-itL}\|_{X^{p_{0},2}\to X^{p_{0},2}}
    \|R^{\beta}\|_{X^{p_{0},2}\to L^{p_{0}}}
    \lesssim
    (1+|t|)^{n(1/p_{0}-1/2)}
  \end{equation*}
  and by the embedding $X^{p_{0},2}\subset L^{p_{0}}$
  we obtain
  \begin{equation*}
    \|R^{2k+\beta}e^{-itL}\|_{L^{p_{0}}\to L^{p_{0}}}
    \lesssim
    (1+|t|)^{n(1/p_{0}-1/2)}.
  \end{equation*}
  Finally, by duality and interpolation, we obtain the second claim.
\end{proof}

We can now conclude the proof of our main results
(Theorems \ref{the:second} and \ref{the:third}):

\begin{thm}[]\label{the:final}
  Assume that $L$ satisfies {\bf (L)}. 
  Let $p\in [p_0,p'_0]$ and $s=n\left|\f{1}{2}-\f{1}{p}\right|$. 
  Then we have the following estimate
  \begin{equation*}    
    \|e^{-itL}\varphi(L)f\|_{p}\les (1+|t|)^s\|f\|_p,
    \qquad t\in \mathbb{R},
  \end{equation*}
  uniformly for $\varphi$ in bounded subsets of 
  $\mathscr{S}(\mathbb{R})$
  (or, more generally, in bounded subsets for the norm
  \eqref{eq:weightedsob}).

  If  Assumption \textbf{(L$_{1}$)} holds,
  we have
  \begin{equation*}
    \|e^{-itL}\varphi(\theta L)f\|_{p}\les 
    (1+\theta^{-1}|t|)^s\|f\|_p,
    \qquad \theta>0, \quad t\in \mathbb{R},
  \end{equation*}
  and the estimate is uniform for $\theta$ in bounded subsets of
  $(0,+\infty)$
  and $\varphi$ in bounded subsets of $\mathscr{S}(\mathbb{R})$
  (or, more generally, in bounded subsets for the norm
  \eqref{eq:weightedsob}).
  If in addition we assume $\kappa=n$ and $M_{0}=0$,
  the estimate is uniform also for all $\theta>0$.
\end{thm}

\begin{proof}
  For the first claim it is sufficient to write
  \begin{equation*}
    e^{-itL}\varphi(L)=
    (I+L)^{2k+\beta}e^{-itL} \cdot
    (I+L)^{-2k-\beta}\varphi(L)
  \end{equation*}
  and use the previous Lemma and Lemma \ref{lem:phiL2}. 
  The second claim is proved by a rescaling argument exactly as
  in the proof of Theorem \ref{the:phiLscal}.
\end{proof}

\begin{proof}[Proof of Theorem \ref{the:fourth}:] Since the proof
is quite similar to that of Theorem \ref{the:third}, 
we just sketch the main steps.

Denote by $p^{\mathcal{X}}_t(x,y)$ the kernel of 
$1_\Omega e^{-tL}1_\Omega$, regarded as an operator on 
functions defined on the entire space $\mathcal{X}$.
Then it is easy to see that
\[
p^{\mathcal{X}}_t(x,y)=\begin{cases}
p_t(x,y), \ \ &\text{if $x,y\in \Omega$},\\
0, \ \ &\text{otherwise}.
\end{cases}
\]
This, along with \eqref{Assumption-Subdomain}, implies
\begin{equation}\label{eq-ptX}
|p^{\mathcal{X}}_t(x,y)|\leq \f{Ce^{M_0t}}{\mu(B(x,t{^{1/m}}))}\exp\Big(-\f{d(x,y)^{m/(m-1)}}{ct^{1/(m-1)}}\Big)
\end{equation}
for all $t>0$ and $x,y\in \mathcal{X}$.

As a consequence, the assumption \eqref{eq1-assumption} and \eqref{eq2-assumption} hold true with $m_1=m_2=m$ and  $1_\Omega e^{-tL}1_\Omega$ taking place of $e^{-tL}$.
	
Arguing similarly to the proof of Proposition \ref{pro:weakerL},  $\nu\in\mathbb{Z}$ with $2^{-\nu}\leq t^{1/m}<2^{-\nu+1}$, there exist $C\ge0$ such that
	for all $t$ and $\nu$ as above we have
	\begin{equation}\label{eq2-LOmega}
	\|{\rm Ad}^{k}_{x}(1_\Omega e^{-tL}1_\Omega)\|_{2\to2}\le
	Ce^{M_0t}2^{-k\nu},
	\qquad
	0\le k\le\lfloor n/2\rfloor+1,\qquad x\in \mathcal{X}.
	\end{equation}
We then argue as in the proof of Proposition \ref{pro:commR} to find that for all $M>M_0$ we have,
	for all $z\in \mathcal{X}$ and
	$0\le k\le\lfloor n/2\rfloor+1$,
	\begin{equation}\label{eq3-LOmega}
	\|{\rm Ad}^k_{z}(1_\Omega(L+M)^{-1}1_\Omega)\|_{2\to 2}
	\lesssim
	(M-M_0)^{-1-\f{k}{m}}
	\end{equation}
	with a constant independent of $z,M$, where $w(x,y)=d(x,y)$.

The argument used in the proof of Proposition \ref{prop-boundednessofRgamma} allows us to obtain that for  $M>M_0$ and
	$\gamma=\f{n}{2m}+\epsilon$,
	with $\epsilon>0$. Then
	$$
	\textstyle
	\|1_\Omega(M+L)^{-\gamma}1_\Omega f\|_{X^{1,2}}\le
	C
	\Bigl(
	\epsilon ^{-1}
	+
	(M-M_0)
	^{\gamma+\f{n-\kappa}2{m}}
	\Bigr)
	\|f\|_{1}.
	$$
Fix $M_1>M_0$ and set $R=(M_1+L)^{-1}$. Then we can verify that 
\begin{equation*}
{\rm Ad}_{z}(1_\Omega e^{-i \xi R}1_\Omega)=
-i\int_{0}^{\xi}
1_\Omega e^{-is R}{\rm Ad}_{z}(1_\Omega R 1_\Omega)e^{-i(\xi-s)R} 1_\Omega ds.
\end{equation*}
Hence, similarly to Lemma \ref{lem:eixiR}, we obtain
	\begin{equation*}
	\|1_\Omega e^{-i \xi R} 1_\Omega f\|_{X^{1,2}}\le
	C(1+|\xi|)^{n/2}\|f\|_{_{X^{1,2}}},
	\qquad
	\xi\in \mathbb{R}.
	\end{equation*}
This, along with the identity
\begin{equation*}
1_\Omega \psi(R)1_\Omega =(2\pi)^{-1}\int 1_\Omega e^{i \xi R}1_\Omega \widehat{\psi}(\xi)d\xi,
\end{equation*}
implies that
	\begin{equation}\label{eq4-LOmega}
	\|1_\Omega \psi(R)1_\Omega f\|_{X^{1,2}}\le
		C\|(1+|\xi|)^{n/2}\widehat{\psi}(\xi)\|_{L^{1}} 
	\|f\|_{X^{1,2}}
	\end{equation}
for any sufficiently smooth funtion $\psi$ on $\mathbb{R}$.

Arguing similarly as in Theorem \ref{the:phiLscal},  for all $\theta>0$ we have
	\begin{equation*}
	\|1_\Omega\varphi(\theta L)1_\Omega f\|_{L^{p}}\le C\|f\|_{L^{p}}
	\end{equation*}
	and the estimate is uniform for 
	$\varphi$ in bounded subsets of $\mathscr{S}(\mathbb{R})$.	Moreover, if  $\kappa=n$, then
	the estimate is uniform for all $\theta>0$.

As this stage, arguing, mutatis mutandis, as in the proof of Theorem \ref{the:third} we obtain that for any $p\in [1,\vc]$  and $s=n\Big|\f{1}{2}-\f{1}{p}\Big|$, 
\begin{equation*}
\|1_\Omega e^{-itL}\varphi(\theta L)1_\Omega f\|_{p}\les (1+\theta|t|)^s\|f\|_{p},
\qquad t\in \mathbb{R},
\end{equation*}
and the estimate is uniform for $\varphi$ in bounded subsets
of $\mathscr{S}(\mathbb{R})$ and $0<\theta\le \theta_{0}$,
for any fixed $\theta_{0}>0$. If, in addition,
$\kappa=n$ and $M_0$, then
the estimate is uniform for all $\theta>0$. This completes our proof.

\end{proof}
\section{Applications}\label{sec:appl}

Our framework is sufficiently general to include a large variety
of applications; in this section we survey a few of the most
interesting cases.

\subsection{Laplace-Beltrami operators with a 
Gaussian heat kernel bound}\label{Subsec4.1}

Let $\mathcal{X}$ be  a  complete  connected  non-compact  
$n$-dimensional
Riemannian  manifold. The  geodesic distance and the 
Riemannian measure are denoted by $d$ and $\mu$, respectively.  
The Laplace-Beltrami operator $L = -\Delta$ on $\mathcal{X}$ 
is nonnegative and self--adjoint. 

We assume that the Riemannian measure $\mu$ satisfies the volume
doubling property \eqref{doublingpro} and the {\it non-collapsing}
condition
\begin{equation}\label{non-collapsing-app1bis}
  \mu(B(x,1))\geq c
\end{equation}
for all $x\in \mathcal{X}$ and for some fixed constant $c>0$.

It is well-known (see \cite{LY}) that if the Ricci
curvature of $\mathcal{X}$ is non-negative, then the heat kernel 
of the heat semigroup $e^{-tL}$ satisfies the estimate
\begin{equation}\label{GU-Laplace}
  e^{-tL}(x,y)\les \f{1}{\mu(B(x,\sqrt{t}))}
  \exp\Big(-\f{d(x,y)^2}{ct}\Big).
\end{equation}
It can be verified that the Gaussian upper bound 
\eqref{GU-Laplace} implies \eqref{eq2-assumption}. Moreover, 
the upper bound \eqref{GU-Laplace} also yields that for 
$\nu\in\mathbb{Z}$ and $2^{-\nu}\leq t^{1/2}<2^{-\nu+1}$, we have
$$
  \sum_{Q\in \mathcal{D}_\nu}\|1_Q e^{-tL}1_{Q'}
  \|_{1\to \vc}\leq C\mu(Q')^{-1}, 
  \ \text{for all $Q'\in \mathcal{D}_\nu$}.
$$
This, in combination with the non-collapsing condition and 
\eqref{doublingpro1}, implies that
$$
  \sum_{Q\in \mathcal{D}_\nu}\|1_Q e^{-tL}1_{Q'}
  \|_{1\to \vc}\leq C(1+ 2^{\nu n}), 
  \ \text{for all $Q'\in \mathcal{D}_\nu$}
$$
and this proves \eqref{eq1-assumption} and \eqref{eq2-assumption}.

Hence, Assumption \textbf{(L$_0$)} is satisfied with 
$m_1=m_2=2$ and $p_0=1$.

\subsection{Laplace-Beltrami operators without Gaussian heat 
kernel bound}

Let $\mathcal{X}$ be  a  complete  connected  non-compact  
Riemannian  manifold. The  geodesic distance and the Riemannian 
measure are denoted by $d$ and $\mu$, respectively. We assume 
that the Riemannian measure $\mu$ satisfies the volume 
doubling property \eqref{doublingpro} and the {
\it non-collapsing} condition \eqref{non-collapsing-app1}.

Let $L=-\Delta$ be the
non-negative Laplace-Beltrami operator on $\mathcal{X}$. We 
assume that the kernel $e^{-tL}(x,y)$ of the semigroup 
$e^{-tL}$ satisfies the following sub-Gaussian heat kernel 
upper estimate with exponent $m>0$
\begin{equation}\label{subGassianEs}
  e^{-tL}(x,y)\leq 
  \begin{cases}
    \f{C}{\mu(B(x,\sqrt{t}))}\exp\Big(\f{d(x,y)^2}{ct}\Big), 
    & 0<t<1
    \\
    \f{C}{\mu(B(x,t^{1/m}))}
    \exp\Big(\f{d(x,y)^{m/(m-1)}}{ct^{1/(m-1)}}\Big), 
    & t\geq 1
  \end{cases}
\end{equation}
for all $x,y\in \mathcal{X}$.

Typical examples that satisfy  \eqref{doublingpro}, 
\eqref{non-collapsing-app1} and \eqref{subGassianEs} include 
certain fractal manifolds and infinite connected locally 
finite graphs. For further details, we refer to \cite{BCG,CCFR}.

By a similar argument as in Subsection \ref{Subsec4.1} one can 
prove that $L$ satisfies Assumption {\bf{(L$_0$)}} with 
$m_1=2$, $m_2=m$ and $p_0=1$.

\subsection{Sierpinski gasket SG in \texorpdfstring{$\mathbb{R}^n$}{}}

Let $\mathcal{X}$ be the unbounded Sierpinski gasket SG in 
$\mathbb{R}^n$. Let $d$ be the induces metric on SG and $\mu$ 
be the Hausdorff measure on SG of dimension $\alpha=\log_2(n+1)$. 
It is well-known that the Hausdorff measure $\mu$ satisfies 
the doubling property \eqref{doublingpro}; moreover,  
\begin{equation}\label{volume-fractal}
  \mu(B(x,r))\les r^\alpha,
\end{equation}
for all $x\in \mathcal{X}$ and $r>0$.

It was also proved in \cite{Ba} that SG admits a local Dirichlet 
form $\mathcal{E}$ which generates a nonnegative 
self--adjoint operator $L$; moreover, the kernel $e^{-tL}(x,y)$ 
of $e^{-tL}$ satisfies the sub-Gaussian estimate
$$
  e^{-tL}(x,y)\les \f{1}{t^{\alpha/m}}
  \exp\Big(-\f{d(x,y)^{m/(m-1)}}{ct^{1/(m-1)}}\Big)
$$
where $m=\log_2(n+3)$ is called the \emph{walk dimension}.

Note that the assumption \eqref{eq2-assumption} is a direct 
consequence of the kernel upper bound above whereas 
the assumption \eqref{eq1-assumption} is a consequence 
of the same kernel upper bound, 
the doubling property \eqref{doublingpro} and 
\eqref{volume-fractal}. Theorefore, $L$ satisfies 
Assumption {\bf{(L$_0$)}} with $m_1=m_2=m$ and  $p_0=1$.

\subsection{Homogeneous groups}

Let ${\bf G}$  be a Lie group of polynomial growth and let 
$X_1, ..., X_k$ be a system of 
left-invariant vector fields on ${\bf G}$ satisfying the 
H\"ormander condition. We define
the Laplace operator $L$ on $L^2({\bf G})$ by 
\begin{equation}\label{defnL-homogeneousgroups}
  L=-\sum_{i=1}^k X_i^2.
\end{equation}  
Denote by $d$ the distance associated with the system $X_1, ..., X_k$,
and let $B(x,r)$ be the corresponding balls.
Then (see \cite{VSC}) there exist positive numbers 
$d, D\geq 0$ such that
\begin{equation}\label{volume-homogroups}
  \mu(B(x,r))\sim
  \begin{cases}
    r^{d}, \  \ \ &r\leq 1\\
    r^{D},\  \ \ &r> 1.
  \end{cases}
\end{equation}
Hence $({\bf G}, d,\mu)$ satisfies the doubling property 
\eqref{doublingpro}.

${\bf G}$ is called a \emph{homogeneous group} (see \cite{FS})
if there exists a \emph{family of dilations} 
$({\delta }_t)_{t>0}$ on ${\bf G}$,
that is to say, a one-parameter group
(${\delta }_t \circ {\delta }_t={\delta }_{t s}$)
of automorphisms of ${\bf G}$ determined by
\begin{equation}\label{e6.2}
  {\delta }_t Y_j =t^{d_j} Y_j ,
\end{equation} 
where $Y_1, ..., Y_{\ell}$ is a linear basis of the Lie algebra of 
${\bf G}$ and $d_j\geq 1$ for $1\leq j\leq \ell$. 
We say that the operator $L$ defined by 
\eqref{defnL-homogeneousgroups}   is \emph{homogeneous} if  
${\delta }_t X_i =t Y_i $ for $1\leq i\leq k$. 
It well known that
the heat kernel of the heat semigroup $e^{-tL}$ satisfies 
the estimate
$$
  e^{-tL}(x,y)\les \f{1}{\mu(B(x,\sqrt{t}))}
  \exp\Big(-\f{d(x,y)^2}{ct}\Big).
$$
This upper bound together with \eqref{volume-homogroups} implies 
that $L$ satisfies \eqref{eq1-assumption} and \eqref{eq2-assumption} with $m_1=m_2=2$ and  
$p_0=1$, and hence $L$ satisfies Assumption {\bf{(L$_0$)}} with $m_1=m_2=2$ and  
$p_0=1$.

\subsection{Bessel operators}

Let $\mathcal{X}=((0,\vc)^m, d\mu(x))$ where 
$d\mu(x)=d\mu_1(x_1)\ldots d\mu_n(x_m)$
and $d\mu_k = x_k^{\al_k}dx_k$, $\alpha_k >-1$, for 
$k=1,\ldots, m$ ($dx_{j}$ being the one dimensional Lebesgue measure).
We endow $\mathcal{X}$ with the distance $d$ defined for
$x=(x_1,\ldots,x_m)$ and $y=(y_1,\ldots,y_m) \in \mathcal{X}$ as
\begin{equation*}
  d(x,y):=|x-y|=\Big(\sum_{k=1}^m|x_k-y_k|^2\Big)^{1/2}.
\end{equation*}
Then it is clear that 
\begin{equation}\label{volume-Bessel}
  \mu(B(x,r))\sim r^m
  \prod_{k=1}^m (x_k+r)^{\alpha_k}.
\end{equation}
Note that this estimate implies the doubling property \eqref{doublingpro} with $n=m+\al_1+\ldots+\al_n$ and the non-collapsing condition \eqref{non-collapsing-app1}.

For an element $x\in \mathbb{R}^m$, unless specified otherwise,
we shall write $x_k$ for the $k$-th component of $x$, 
$k=1,\ldots,m$. Moreover, for $\lambda\in \mathbb{R}^m$, 
we write $\lambda^2=(\lambda^2_1,\ldots, \lambda_m^2)$.

We consider the second order Bessel differential operator
$$
  L=-\Delta-\sum_{k=1}^m \f{\al_k}{x_k}\f{\partial}{\partial x_k}
$$
whose system of eigenvectors is defined by
$$
  E_\lambda(x):= \prod_{k=1}^nE_{\lambda_k}(x_k),
  \quad
  E_{\lambda_k}(x_k):=
  (x_k\lambda_k)^{-(\al_k-1)/2}J_{(\al_k-1)/2}
  (x_k\lambda_k),
  \quad
  \lambda, x\in \mathcal{X}
$$
where $J_{(\al_k-1)/2}$ is the Bessel function of the first 
kind of order $(\al_k-1)/2$ (see \cite{L}). 
It is known that $L(E_\lambda)=|\lambda|^2 E_{\lambda}$. 
Moreover, the functions $E_{\lambda_k}$ are  eigenfunctions 
of the one-dimension Bessel operators
$$
  L_k=-\f{\partial^2}{\partial x_k\,^2}-\f{\al_k}
  {x_k}\f{\partial}{\partial x_k}
$$
and indeed $L_k(E_{\lambda_k})=\lambda_k^2 E_{\lambda_k}$ for 
$k=1,\ldots, m$.

It is well known that $L$ is nonnegative and self--adjoint; 
moreover, the kernel $e^{-tL}(x,y)$ of $e^{-tL}$ satisfies 
the Gaussian estimate
\begin{equation}\label{Gaussian-BesselOperator}
  e^{-tL}(x,y)\les \f{1}{\mu(B(x,\sqrt{t}))}
  \exp\Big(-\f{d(x,y)^2}{ct}\Big).
\end{equation}

Hence, the Gaussian upper bound \eqref{Gaussian-BesselOperator}, along with the doubling and the non-collapsing 
properties imply Assumption {\bf{(L$_0$)}} with 
$m_1=m_2=m$ and  $p_0=1$.

\subsection{Schr\"odinger operators with real potentials on manifolds}

Let $\mathcal{X}$ be  a  complete  connected  non-compact  
Riemannian  manifold. The  geodesic distance and the Riemannian 
measure are denoted by $d$ and $\mu$, respectively.  We assume that 
the Riemannian measure $\mu$ satisfies the doubling property 
\eqref{doublingpro} and the non-collapsing condition 
\eqref{non-collapsing-app1}.
We also assume that the heat kernel $p_t(x,y)$ of the
Laplace-Beltrami operator $-\Delta$ satisfies the standard 
Gaussian upper bound
\begin{equation}\label{GaussianUperbound}
p_t(x,y)\leq \f{C}{\mu(B(x,\sqrt{t}))}\exp\Big(-\f{d^2(x,y)}{ct}\Big).
\end{equation}

We now consider the Schr\"odinger operator 
$L=-\Delta+V$, $V\in L^1_{{\rm loc}}(\mathcal{X})$.
If the potential $V$ is nonnegative, then the  kernel of
the semigroup $\{e^{-tL}\}_{t>0}$ generated by $L$ satisfies 
the same Gaussian bound (\ref{GaussianUperbound}); 
in the general case, we must impose some conditions on the
negative part of $V$.
Denote by $V^+$ and $V^-$ the positive and negative parts 
of $V$, respectively. We define
$$
\mathcal{Q}(u,v)=\int_{\mathcal{X}}\nabla u\nabla v d\mu + \int_{\mathcal{X}}V^+ uvd\mu-
\int_{\mathcal{X}}V^- uvd\mu
$$
with domain
$$
\mathcal{D}(\mathcal{Q})=\{u\in W^{1,2}(\mathcal{X}): \int_{\mathcal{X}}V^+u^2 d\mu <\vc\}.
$$
Then we assume that the positive part $V^+\in L^1_{{\rm{loc}}}$ and  the negative part $V^-$ satisfy the following condition
\begin{equation}\label{AssumptiononV}
  \int_{\mathcal{X}} V^-u^2d\mu\leq 
  \alpha\Big[\int_\mathcal{X}|\nabla u|^2d\mu 
    +\int_\mathcal{X}
  V^+u^2d\mu\Big], \forall u \in \mathcal{D}(\mathcal{Q}),
\end{equation}
for some $\alpha\in (0,1)$.
  
It was proved in \cite[Theorem 3.4]{AO} that for any 
$(\f{2}{1-\sqrt{1-\alpha}})'<p_0<2$ there exist $C, c>0$ 
and $\beta>0$ such that  
$$
  \|1_{B(x,r)}e^{-sL}1_{B(y,r)}\|_{p_0\to p_0'}\leq 
  C\mu(B(x,r))^{-\f{1}{p_0}+\f{1}{p'_0}}\Big(\max
  \Big(\f{r}{\sqrt{s}},\f{\sqrt{s}}{r}\Big)
  \Big)^\beta \exp\Big(-\f{{\rm dist}(B(x,r),B(y,r))^2}{ct}\Big)
$$
for all $r,s>0$ and $x,y\in \mathcal{X}$.

This, in combination with the volume doubling property 
\eqref{doublingpro} and the non-collapsing condition \eqref{non-collapsing-app1}, implies that 
Assumption {\bf{(L$_0$)}} is satisfied with $m_1=m_2=2$ and any 
$(\f{2}{1-\sqrt{1-\alpha}})'<p_0<2$.

\subsection{Schr\"odinger operators with inverse-square potentials}

Consider the following Schr\"odinger operators with inverse square
potential on $\mathbb{R}^n$, $n\geq 3$:
\begin{equation}\label{defn-La}
  \mathcal{L}_a = -\Delta+\f{a}{|x|^2} \quad 
  \text{with} \quad a\geq -\Big(\f{n-2}{2}\Big)^2.
\end{equation}
Set
$$
\sigma:=\f{n-2}{2}-\f{1}{2}\sqrt{(n-2)^2+4a}.
$$
The Schr\"odinger operator  $\mathcal{L}_a$ is understood as the 
Friedrichs extension of $-\Delta+\f{a}{|x|^2}$  defined initially 
on $C^\vc_c(\mathbb{R}^n\backslash\{0\})$. The condition 
$a\geq -\Big(\f{n-2}{2}\Big)^2$ guarantees that $\mathcal{L}_a$ 
is nonnegative.
It is well-known that $\mathcal{L}_a$ is self--adjoint and the 
extension may not be unique as 
$-\Big(\f{n-2}{2}\Big)^2\le a<1-\Big(\f{n-2}{2}\Big)^2$. 
For further details, we refer the readers to \cite{Ka, PST, T}.
For the corresponding heat kernel, we have the following result:

\begin{thm}[\cite{MS, LS}]\label{thm-heatkernelLa}
  Assume $n\geq 3$ and $a\geq -\Big(\f{n-2}{2}\Big)^2$. 
  Then there exist two positive constants $C$ and $c$ such 
  that for all $t>0$ and $x,y \in \mathbb{R}^n\backslash \{0\}$,
  $$
    p_t(x,y)\leq C\Big(1+\f{\sqrt{t}}{|x|}
    \Big)^\sigma\Big(1+\f{\sqrt{t}}{|y|}\Big)^\sigma 
    t^{-n/2}e^{-\f{|x-y|^2}{ct}}.
  $$
\end{thm}

Set $n_\sigma=n/\sigma$ if $\sigma>0$ and $n_\sigma=\vc$ if 
$\sigma\leq 0$. From Theorem \ref{thm-heatkernelLa} and 
Theorem 3.1 in \cite{DA}, for any 
$n'_\sigma <p\leq q< n_\sigma$ there exist $C,c>0$ such that 
for every $t> 0$, any measurable subsets 
$E, F\subset \mathbb{R}^n$, and all $f\in L^p(E)$, we have:
\begin{equation}\label{eq1-Tt}
  \left\|e^{-tL}f\right\|_{L^q(F)}\leq 
  Ct^{-\f{n}{2}(\f{1}{p}-\f{1}{q})}e^{-\f{d(E,F)^2}{ct}}
  \|f\|_{L^p(E)}.
\end{equation} 
Hence, with the standard dyadic systems in $\mathbb{R}^n$, this implies that Assumption {\bf{(L$_0$)}} is satisfied with 
$m_1=m_2=2$ and any $n_\sigma'<p_0<2$. Moreover, in this situation the reverse doubling condition \eqref{RD-eq} is valid with $\kappa=n$.

\subsection{Sub-Laplacian operators on Heisenberg groups}
Let $\mathbb{H}^d$ be a $(2d + 1)$-dimensional Heisenberg group. Recall that a $(2d + 1)$-dimensional Heisenberg group  is a connected and simply connected
nilpotent Lie group with the underlying manifold $\mathbb{R}^{2d}\times \mathbb{R}$. The group structure is defined by
$$
(x, s)(y, t) = (x + y, s + t + 2 \sum_{j=1}^{d}(x_{d+j}y_j - x_jy_{d+j}))
$$
The homogeneous norm on $\mathbb{H}^d$ is defined by
$$
|(x, t)| = (|x|^4 + |t|^2)^{1/4}  \ \text{for all} \  (x, t) \in \mathbb{H}^d.
$$
See for example \cite{St}.

This norm satisfies the the triangle inequality and hence induces a left-invariant metric $d((x, t), (y, s)) = |(-x, -t)(y, s)|$. Moreover,
there exists a positive constant $C$ such that $|B((x, t), r)| = Cr^n$, where $n = 2d + 2$ is
the homogeneous dimension of $\mathbb{H}^d$ and $|B((x, t), r)|$ is the Lebesgue measure of the ball
$B((x, t), r)$. Obviously, the triplet $(\mathbb{H}^d, d, dx)$ satisfies the doubling condition \eqref{doublingpro}, the reverse doubling condition \eqref{RD-eq} with $\kappa=n$, and the {\it non-collapsing} condition \eqref{non-collapsing-app1}.

A basis for the Lie algebra of left-invariant vector fields on $\mathbb{H}^d$ is given by
$$
X_{2d+1}=\f{\partial}{\partial t}, X_{j}=\f{\partial}{\partial x_j}+2x_{d+j}\f{\partial}{\partial t}, X_{d+j}=\f{\partial}{\partial x_{d+j}}-2x_{j}\f{\partial}{\partial t}, \ \ j=1,\ldots,d.
$$
The sub-Laplacian $\Delta_{\mathbb{H}^d}$ is defined by
$$
\Delta_{\mathbb{H}^d}=-\sum_{j=1}^{2d}X_j^2.
$$
Furthermore, it is well-known that the sub-Laplacian $\Delta_{\mathbb{H}^d}$ satisfies the Gaussian upper bound:
$$
e^{-t\Delta_{\mathbb{H}^d}}((x,u),(y,s))\leq \f{C}{t^{n/2}}\exp\Big(-\f{d((x,u),(y,s)^2}{ct}\Big).
$$

In $\mathbb{H}^d$, we consider the standard dyadic system consists of the cubes
$$
2^{-k}((0,1]^{2d}+j)\times 4^{-k}((0,1]+\ell), k\in \mathbb{Z}, \ \ j\in \mathbb{Z}^{2d}, \ell\in \mathbb{Z}.
$$
Hence, the Gaussian upper bound yields the assumption \textbf{(L$_0$)}  with $m_1=m_2=2$ and $p_0=1$.

\subsection{Dirichlet Laplacians on Lipchitz domains}

Let $\mathcal{X}=(\mathbb{R}^n,|\cdot|, dx)$. Then $\mathcal{X}$ is a space of homogeneous type satisfying \eqref{RD-eq} with $\kappa=n$ and the {\it non-collapsing} condition \eqref{non-collapsing-app1}.

Let $\Omega$ be a connected open subset of $\mathbb{R}^n$. Note that $\Omega$ may not satisfy the doubling condition. Let $L=\Delta_\Omega$ be Dirichlet Laplacian on the domain $\Omega$. It is well known that the semigroup kernel 
$p_t(x,y)$ of $e^{-tL}$  satisfies the Gaussian upper bound
$$
  p_t(x,y)(x,y)\leq \f{1}{(4\pi t)^{n/2}}
  \exp\Big(-\f{|x-y|^2}{4t}\Big),
$$ 
for all $t>0$ and all $x,y\in \Omega$.

Hence, all assumptions in Theorem \ref{the:fourth} are satisfied with $\mathcal{X}=(\mathbb{R}^n,|\cdot|, dx)$, $L=\Delta_\Omega$ and $\kappa=n$.

\subsection{Schr\"{o}dinger operators with singular potentials}
\label{sub:sch}

For our last example, we recall the definition of the
\emph{Kato class} $K_{n}$ of potentials. The 
measurable function
$V:\mathbb{R}^{n}\to \mathbb{R}$ belongs to $K_{n}$
if the following conditions are satisfied:
\begin{enumerate}
  \item If $n\ge3$,
  \begin{equation*}
    \textstyle
    \lim_{\alpha \downarrow0}
    \sup_{x}
    \int_{|x-y|\le \alpha}
    |x-y|^{2-n}V(x)dx=0.
  \end{equation*}
  \item If $n=2$,
  \begin{equation*}
    \textstyle
    \lim_{\alpha \downarrow0}
    \sup_{x}
    \int_{|x-y|\le \alpha}
    \log(|x-y|^{-1})V(x)dx=0.
  \end{equation*}
  \item If $n=1$,
  \begin{equation*}
    \textstyle
    \sup_{x}
    \int_{|x-y|\le 1}V(x)dx<\infty.
  \end{equation*}
\end{enumerate}
Moreover, we say that $V\in K_{n,loc}$
if ${\bf 1}_{B}V\in K_{n}$ for all balls $B$.

We consider a Schr\"{o}dinger operator
of the form $L=-\Delta+V(x)$ on $\mathbb{R}^{n}$, $n\ge1$.
We assume that the positive part $V_{+}$ of $V$ is in $K_{n,loc}$
while the negative part $V_{-}$ is in $K_{n}$. Then the results of
\cite{Simon82-a} (see in particular Proposition B.6.7)
imply that $L$ can be realized as a
semibounded self--adjoint operator in $L^{2}(\mathbb{R}^{n})$,
and that the heat kernel $e^{-tL}$ satisfies
\begin{equation}\label{eq:gaussab}
  |e^{-tL}(x,y)|\le
  C t^{-n/2}e^{M_{0}t}e^{-\frac{|x-y|^{2}}{ct}}
\end{equation}
with $C,c>0$. Thus Assumption \textbf{(L$_{1}$)}
is satisfied, with $M_{0}\ge0$. If in addition we assume that
the negative part satisfies
\begin{equation}\label{eq:Vmeno}
  \textstyle
  \sup_{x}\int|x-y|^{2-n}V_{-}(y)dy<2\pi^{n/2}/\Gamma(n/2-1)
\end{equation}
in dimension $n\ge3$ (or $V_{-}=0$ in dimensions 1,2)
then in  \cite{DAnconaPierfelice05-a} it is proved that
one can take $M_{0}=0$, so that the uniform estimates
of Theorem \ref{the:third} apply.

Moreover, one can consider the same operator $L$ with
Dirichlet boundary conditions on $L^{2}(\Omega)$,
for an open subset $\Omega$ of $\mathbb{R}^{n}$.
If we assume for simplicity $V\ge0$, then by the maximum
principle we obtain that the heat kernel is nonnegative
and satisfies again the upper Gaussian estimate
\eqref{eq:gaussab}, with $M_{0}=0$ i.e. all the assumptions
of the second part of Theorem \ref{the:fourth} are
satisfied.

Similar results can be proved for the magnetic Schr\"{o}dinger
operators of the form $(i \nabla+A(x))^{2}+V(x)$,
using the heat kernel estimates proved in
\cite{DAnconaFanelliVega10-a}, and for elliptic operators
with fully variable coefficients on exterior domains,
via the results of \cite{CassanoDAncona15-a}. We omit the
details.

\subsection{Magnetic Schr\"odinger operator}\label{Subsec11}
Consider the magnetic Schr\"odinger operator on $\mathbb{R}^n$ defined by
\begin{equation*}
L=(i \nabla+A(x))^{2}+V(x),
\end{equation*}
with magnetic potential
$A=(A_{1},\dots,A_{n})$ and electric potential $V(x)$.

If we choose as weight function
$w(x,y)=x-y:\mathbb{R}^{2n}\to \mathbb{R}^{n}$ and $\mathscr D(w)=C_0^\vc(\mathbb{R}^n)$, then we have 
	\begin{equation*}
	\textstyle
	{\rm Ad}^{1}_{x}(L)=
	2\nabla+2iA
	,\qquad
	{\rm Ad}^{2}_{x}(L)=(2,\dots,2)
	\end{equation*}
	and
	\begin{equation*}
	{\rm Ad}^{j}_{x}(L)=0
	\ \text{for}\ 
	j\ge2.
	\end{equation*}
	The vector of operators ${\rm Ad}^{2}_{x}(L)R$
	is obviously bounded on $L^{2}$; since $R$ is also bounded
	from $L^{2}(\mathbb{R}^{n})$ to $H^{1}(\mathbb{R}^{n})$,
	if the magnetic potential satisfies
	\begin{equation*}
	\|Af\|_{L^{2}}\lesssim\|f\|_{H^{1}}
	\end{equation*}
	then also ${\rm Ad}^{1}_{x}(L)R$ is bounded on $L^{2}$. Moreover, by elementary computations one can write
	${\rm Ad}^{k}_{x}(R)$ as a linear combination of terms
	\begin{equation}\label{eq:adR}
	R \;{\rm Ad}^{k_{1}}_{x}(L)\;
	R \;{\rm Ad}^{k_{2}}_{x}(L)\;
	\dots
	R \;{\rm Ad}^{k_{N}}_{x}(L)\;R
	\end{equation}
	with $k\ge k_{i},N\ge1$ and $k_{1}+ \dots +k_{N}=k$. For instance, in dimension $n\ge3$ it is sufficient to assume
	that $|A|\le C+ C|x|^{-1}$, thanks to Hardy's inequality.
	
	As a consequence, it follows the condition \textbf{(L)}.


\end{document}